\documentclass[12pt]{amsart}
\usepackage{amsmath}
\usepackage{amssymb}
\usepackage{epsfig}
\usepackage{graphics}
\usepackage{amsthm}
\usepackage{bm}
\usepackage{url}

\usepackage{color}

\numberwithin{equation}{section}

\newtheorem{theorem}{Theorem}[section]
\newtheorem{lemma}[theorem]{Lemma}
\newtheorem{corollary}[theorem]{Corollary}
\newtheorem{proposition}[theorem]{Proposition}
\newtheorem{Definition}[theorem]{Definition}

\newtheorem{Rem}[theorem]{Remark}
\newenvironment{remark}{\begin{Rem}\rm}{\end{Rem}}

\newcommand\R{\mathbb{R}}

\newcommand\N{\mathbb{N}}

\newcommand\be{\begin{equation}}
\newcommand\ee{\end{equation}}
\newcommand\bea{\begin{eqnarray}}
\newcommand\eea{\end{eqnarray}}
\newcommand\beaa{\begin{eqnarray*}}
\newcommand\eeaa{\end{eqnarray*}}
\newcommand{\baa}{\begin{array}}
\newcommand{\eaa}{\end{array}}
\newcommand\beba{\begin{equation}\left\{\begin{array}{rcl}}
\newcommand\eeba{\end{array}\right.\end{equation}}
\newcommand\bebaa{\begin{equation*}\left\{\begin{array}{rcl}}
\newcommand\eebaa{\end{array}\right.\end{equation*}}
\newcommand\beca{\begin{equation}\left\{\begin{array}{rcll}}
\newcommand\eeca{\end{array}\right.\end{equation}}
\newcommand\becaa{\begin{equation*}\left\{\begin{array}{rcll}}
\newcommand\eecaa{\end{array}\right.\end{equation*}}
\newcommand\beda{\begin{equation}\left\{\begin{array}{l}}
\newcommand\eeda{\end{array}\right.\end{equation}}

\newcommand{\ep}{\varepsilon}
\def\epsilon{\varepsilon}
\def\tilde{\widetilde}

\begin{document}
\title[Positive entire solutions]{Localized and expanding entire solutions of reaction-diffusion equations}

\author{F. Hamel}
\address[F. Hamel]{Aix Marseille Univ, CNRS, Centrale Marseille, I2M, Marseille, France}
\email{francois.hamel@univ-amu.fr}

\author{H. Ninomiya}
\address[H. Ninomiya]{School of Interdisciplinary Mathematical Sciences, Meiji University
4-21-1 Nakano, Nakano-ku, Tokyo 164-8525, Japan}
\email{hirokazu.ninomiya@gmail.com}

\begin{abstract}
This paper is concerned with the spatio-temporal dynamics of nonnegative bounded entire solutions of some reaction-diffusion equations in $\R^N$ in any space dimension~$N$. The solutions are assumed to be localized in the past. Under certain conditions on the reaction term, the solutions are then proved to be time-independent or heteroclinic connections between different steady states. Furthermore, either they are localized uniformly in time, or they converge to a constant steady state and spread at large time. This result is then applied to some specific bistable-type reactions.
\end{abstract}


\maketitle

\begin{center}
{\it In memory of Genevi\`eve Raugel, with admiration and respect}
\end{center}



\section{Introduction and the main result}\label{sec:1}
\setcounter{equation}{0}

In this paper we are concerned with nonnegative bounded entire solutions of the following reaction-diffusion equation:
\be\label{eq:ACN}
u_t=\Delta u+f(u),\ \ t\in\R,\ \ x\in\R^N,
\ee
where $f:[0,+\infty)\to\R$ is a $C^1$ function such that
\be\label{hypf}
f(0)=0\ \hbox{ and }\ f'(0)<0.
\ee

The solutions are always understood in the classical sense $C^{1,2}_{t,x}(\R\times\R^N)$, from the parabolic regularity theory. Notice immediately that, for a nonnegative bounded solution $u$ of~\eqref{eq:ACN}, either $u(t,x)=0$ for all $(t,x)\in\R\times\R^N$, or $u(t,x)>0$ for all $(t,x)\in\R\times\R^N$, from the strong parabolic maximum principle and the uniqueness of the bounded solutions for the associated Cauchy problem.

The solutions $u$ are called {\it entire} as they are defined for all $t\in\R$ and $x\in\R^N$. The solution~$u$ of \eqref{eq:ACN} is called {\it bounded} if $u$ is an entire solution of~\eqref{eq:ACN} and 
\[
\sup_{(t,x)\in\R^{N+1}}|u(t,x)|<\infty.
\]
We are especially interested in the description of their limit profiles as $t\to\pm\infty$. If a solution~$u$ converges, in some sense to be made precise, to some limit states $\phi_\pm$ as $t\to\pm\infty$, then~$u$ is a heteroclinic connection between $\phi_-$ and $\phi_+$ if $\phi_-\neq\phi_+$, while it is homoclinic to $\phi_\pm$ if~$\phi_-=\phi_+$ (we will actually prove that the homoclinic connections reduce to time-independent solutions under the assumptions in this paper). The description and the pro\-perties of the entire solutions of~\eqref{eq:ACN} are of particular importance such as, for any element~$\varphi$ of the $\omega$-limit set of any nonnegative initial condition of the associated Cauchy problem giving rise to a bounded global solution, and for any $t_0\in\R$, there is a bounded entire solution $u$ of~\eqref{eq:ACN} such that~$u(t_0,\cdot)=\varphi$.


\subsection{Localized solutions in the past and localized steady states}

We are interested in solutions that are localized in the past, in the sense that
\be\label{localized}
u(t,x)\to0\ \hbox{ as }|x|\to+\infty,\ \hbox{ uniformly in }t\le0.
\ee
Throughout the paper, $x\mapsto|x|$ denotes the Euclidean norm in $\R^N$, $(x,y)\mapsto x\cdot y$ denotes the Euclidean inner product, $B(x,R)$ denotes the open Euclidean ball of center $x\in\R^N$ and radius $R>0$, and $B_R=B(0,R)$. Notice that the condition $f(0)=0$ is then forced by~\eqref{localized}. Furthermore, from standard parabolic estimates and the boundedness of $u$, condition~\eqref{localized} is equivalent to $\lim_{|x|\to+\infty}u(t,x)=0$ uniformly in $t\le t_0$ for some (or equivalently for all)~$t_0\in\R$. This, however, does not necessarily mean that $\lim_{|x|\to+\infty}u(t,x)=0$ uniformly in~$t\in\R$ (such solutions are called {\it uniformly localized}), and one of the main features of the paper is to show a dichotomy between the solutions that are uniformly localized and those that spread as $t\to+\infty$.

The description of the positive bounded solutions of~\eqref{eq:ACN} satisfying~\eqref{localized} is closely related to the study of the positive bounded localized steady states $\phi\in C^2(\R^N)$, solving
\be\label{steady}\left\{\baa{l}
\Delta\phi+f(\phi)=0\hbox{ and }\phi>0\hbox{ in }\R^N,\vspace{3pt}\\
\phi(x)\to0\hbox{ as }|x|\to+\infty.\eaa\right.
\ee
Under the condition $f'(0)<0$, it is known~\cite{GNN2,L,LN} that any solution $\phi$ of~\eqref{steady} is radially symmetric and decreasing with respect to its center, namely there exist a point $x_0\in\R^N$ and a~$C^2([0,+\infty))$ function $\Phi$ such that $\Phi'<0$ in $(0,+\infty)$ and
\be\label{profile}
\phi(x)=\Phi(|x-x_0|)\ \hbox{ for all }x\in\R^N.
\ee
It then follows from the strong maximum principle applied to $\phi$ that
\be\label{fPhi}
f(\Phi(0))=f\Big(\max_{\R^N}\phi\Big)>0,
\ee
hence, together with~\eqref{hypf}, there is a unique real number $m_\phi$ such that
\be\label{defmphi}
0<m_\phi<\max_{\R^N}\phi,\ \ f(m_\phi)=0\ \hbox{ and }\ f>0\hbox{ in }\Big(m_\phi,\max_{\R^N}\phi\Big].
\ee
Lastly, since
\be\label{eqradial}
\Phi''(r)+\frac{N-1}{r}\,\Phi'(r)+f(\Phi(r))=0\ \hbox{ for all }r\in(0,+\infty)
\ee
and $\Phi'(0)=\Phi'(+\infty)=0$ (the limit $\Phi'(+\infty)=0$ coming from~\eqref{steady}-\eqref{profile} and standard elliptic estimates), integrating the above equation against $\Phi'$ over $(0,+\infty)$ yields $F(\max_{\R^N}\!\phi)=F(\Phi(0))=0$ if $N=1$ and $F(\max_{\R^N}\!\phi)=F(\Phi(0))>0$ if $N\ge2$, where
$$F(s)=\int_0^sf(\sigma)\,d\sigma\ \hbox{ for $s\ge0$}.$$
Since $f(\max_{\R^N}\!\phi)>0$, there is then $\eta>0$ such that $f>0$ in $[\max_{\R^N}\!\phi-\eta,\max_{\R^N}\!\phi+\eta]$ and~$F>0$ in $(\max_{\R^N}\!\phi,\max_{\R^N}\!\phi+\eta]$ if $N=1$ (resp. in $[\max_{\R^N}\!\phi,\max_{\R^N}\!\phi+\eta]$ if $N\ge2$). In the sequel, we also set
\be\label{defMphi}
M_\phi=\inf\Big\{s\ge\max_{\R^N}\phi:f(s)=0\Big\}\in\Big(\!\max_{\R^N}\phi,+\infty\Big].
\ee
Notice that
\be\label{mMphi}
m_\phi<\max_{\R^N}\phi<M_\phi\ \hbox{ and }\ f>0\hbox{ in }(m_\phi,M_\phi),
\ee
and that $M_\phi$ may be equal to $+\infty$ (we refer to some specific examples in Section~\ref{sec:2}).

Furthermore, not only the steady states of~\eqref{steady} are radially symmetric and decreasing with respect to some center, but so are the bounded entire solutions of~\eqref{eq:ACN} which are localized in the past. Namely, it follows from~\cite{P1} that, for any positive bounded solution $u$ of~\eqref{eq:ACN}-\eqref{localized}, there is a point $x_0\in\R^N$ such that
$$\left\{\baa{ll}
u(t,x)=u(t,y) & \hbox{for all }(t,x,y)\in\R\times\R^N\times\R^N\hbox{ with }|x-x_0|=|y-x_0|,\vspace{3pt}\\
\nabla u(t,x)\cdot(x-x_0)<0 & \hbox{for all }(t,x)\in\R\times\R^N\hbox{ with }x\neq x_0.\eaa\right.$$


\subsection{The main result}

In the following theorem, which is the main result of the paper, we call $\mathcal E$ the set of $C^2(\R^N)$ solutions of~\eqref{steady} and, for any continuous bounded function $\varphi:\R^N\to\R$ and any set $\mathcal A$ of continuous bounded functions, we denote
$${\rm{dist}}(\varphi,\mathcal A)=\inf_{\psi\in\mathcal A}\|\varphi-\psi\|_{L^\infty(\R^N)}.$$

\begin{theorem}\label{th1}
Assume that $f$ satisfies~\eqref{hypf} and
\be\label{hypf2}
F<0\hbox{ in }(0,m_\phi]\ \hbox{ for all }\phi\in \mathcal E.
\ee
If there exists a positive bounded solution $u$ of~\eqref{eq:ACN} satisfying~\eqref{localized},
then $\mathcal E\neq\emptyset$ and
\be\label{conv1}
{\rm{dist}}(u(t,\cdot),\mathcal E)\to0\hbox{ as }t\to-\infty.
\ee
Furthermore,
\begin{enumerate}
\item either $u(t,\cdot)\to0$ uniformly in $\R^N$ as $t\to+\infty$,
\item or there is $\phi\in\mathcal E$ such that $u(t,\cdot)\to\phi$ uniformly in $\R^N$ as $t\to+\infty$,
\item or else there is a continuous function $\xi:\R\to\R$ depending on $u$ and some positive constants $M$ and $c$ only depending on $f$ such that
\be\label{defxit}\left\{\baa{l}
\displaystyle\limsup_{t\to+\infty}\Big(\max_{|x|\le\xi(t)-A}|u(t,x)-M|\Big)\to0\vspace{3pt}\\
\displaystyle\limsup_{t\to+\infty}\Big(\max_{|x|\ge\xi(t)+A}u(t,x)\Big)\to0\eaa\right.\hbox{ as }A\to+\infty
\ee
and
\be\label{spreading}
\lim_{t\to+\infty}\frac{\xi(t)}{t}=c,
\ee
with $f(M)=0$, $f'(M)\le0$ and $c$ characterized by the existence of a function $\varphi\in C^2(\R)$ solving
\be\label{eqvarphi}
\varphi''+c\varphi'+f(\varphi)=0\hbox{ in }\R,\ \ \varphi'<0\hbox{ in }\R,\ \ \varphi(-\infty)=M,\ \ \varphi(+\infty)=0.
\ee
\end{enumerate}
\end{theorem}

Property~\eqref{conv1} means that the $\alpha$-limit sets, with respect to the uniform convergence in~$\R^N$, of the positive bounded solutions $u$ of \eqref{eq:ACN} consist of steady states solving~\eqref{steady}. As far as the behavior of a solution $u$ as $t\to+\infty$ is concerned, it turns out that, in both cases~(i) and~(ii), $u(t,x)\to0$ as $|x|\to+\infty$ uniformly in $t\in\R$, namely $u$ is then called {\it uniformly localized}. As a consequence, the conclusion of Theorem~\ref{th1} means there is a dichotomy between the uniformly localized solutions and the ones which converge locally uniformly to a positive constant, with a positive spreading rate. Notice that in all cases~(i),~(ii) and~(iii), the solution $u$ converges locally uniformly in $\R^N$ as $t\to+\infty$ to a steady state (either a necessarily non-constant solution of~\eqref{steady}, or the constants $0$ or $M$) and its $\omega$-limit set (with respect to the uniform convergence in cases~(i) and~(ii) and to the locally uniform convergence in case (iii)) is a singleton.

This situation is in contrast with some non-convergence and even non-quasiconvergence results of some positive bounded solutions of the Cauchy problems of the Fujita equation
\be\label{eqfujita}
u_t=\Delta u+u^p,
\ee
for which the $\omega$-limit set (with respect to the locally uniform convergence) of some initial conditions may not be reduced to a single steady state (non-convergence) or may even contain other elements than steady states (non-quasiconvergence). Such results have been proved in~\cite{P4,PY1,PY3} for~\eqref{eqfujita} in high dimensions $N$ for some ranges of values of $p$, even for solutions which are localized at large time (see also~\cite{ER,P2,P3} for further non-quasiconvergence results with non-localized oscillating initial conditions and bistable nonlinearities of the type~\eqref{cond-f} below). On the other hand, convergence or quasiconvergence hold for all functions~$f$ in dimension $N=1$ with compactly supported initial conditions~\cite{DM1} or for generic functions~$f$ in any dimension $N\ge1$ with initial conditions converging to $0$ at infinity~\cite{MP2,MP3}, while the existence of at least one steady state in bounded trajectories has been shown in dimensions~$N\le 2$~\cite{GS}. We refer to~\cite{MZ1,MZ2} for further convergence or quasiconvergence results for some bistable, ignition or monostable nonlinearities $f$ in any dimension $N\ge1$ with radially decreasing initial conditions, and to~\cite{HR1} for a general overview on convergence results for gradient-like parabolic or hyperbolic equations.

\begin{remark}
It actually follows from the proof of Theorem~\ref{th1}, in particular from Steps~4 and~5 in Section~\ref{sec42}, that a similar result as~\eqref{defxit}-\eqref{spreading} holds for the spreading solutions of the associated Cauchy problem with localized initial conditions. More precisely, let~$0<m<M$ be given, and let $f:[0,M]\to\R$ be a given $C^1([0,M])$ function such that~$f(0)=f(m)=f(M)=0$, $f'(0)<0$, $f>0$ in $(m,M)$, $F<0$ in $(0,m]$ and $F(M)>0$. Let $u_0:\R^N\to[0,M]$ be a continuous function such that $\lim_{|x|\to+\infty}u_0(x)=0$ and $u_0\not\equiv0$ in~$\R^N$. Now, if the bounded solution $u$ of the Cauchy problem associated with~\eqref{eq:ACN} with initial condition $u_0$ is assumed to be such that
$$u(t,\cdot)\to M\hbox{ as $t\to+\infty$ locally uniformly in $\R^N$},$$
then properties~\eqref{defxit}-\eqref{spreading} still hold for some continuous function $\xi:[0,+\infty)\to\R$, where~$c$ is characterized by the existence of a solution $\varphi\in C^2(\R)$ of~\eqref{eqvarphi}. Notice that pro\-perty~\eqref{defxit} implies in particular that, for each $0<\epsilon\le M/2$ and each unit vector~$e$, the diameter of the set $\{r\ge0:\epsilon\le u(t,re)\le M-\epsilon\}$ is bounded as~$t\to+\infty$ (see also the second paragraph after Remark~\ref{rem13}). Furthermore, the same conclusion holds if, instead of~$\lim_{|x|\to+\infty}u_0(x)=0$, one assumes that $\limsup_{|x|\to+\infty}u_0(x)\le\eta$, with $\eta>0$ such that~$f<0$ in $(0,\eta]$ (in that case, one has $\limsup_{|x|\to+\infty}u(t,x)\to0$ as~$t\to+\infty$).
\end{remark}


\subsection{Comments about the assumptions~\eqref{hypf} and~\eqref{hypf2} on $f$, and~\eqref{localized} on~$u$}

Let us make in the following paragraphs some comments about the role and necessity of the various assumptions on $f$ and $u$ used in Theorem~\ref{th1}.

Let us first discuss the linear stability assumption~\eqref{hypf} on $f$. Firstly, as already emphasized, the equality $f(0)=0$ is necessary for~\eqref{localized} to hold. Secondly, if $f'(0)=0$, then Theorem~\ref{th1} does not hold in general. For instance, in dimensions~$N\ge3$ and for $(N+2)/(N-2)<p<p_L$ with $p_L=(N-4)/(N-10)$ if $N\ge11$ and $p_L=\infty$ if $N\le 10$, the Fujita equation~\eqref{eqfujita} admits positive bounded entire solutions $u$ which are uniformly localized and are homoclinic to $0$, in the sense that~$\|u(t,\cdot)\|_{L^\infty(\R^N)}\to0$ as~$t\to\pm\infty$, see~\cite{FY}. Thirdly, if $f'(0)>0$, then Theorem~\ref{th1} does not make sense in general. Consider for instance a $C^2$ concave function~$f:[0,+\infty)\to\R$ such that $f(0)=0$, $f'(0)>0$ and~$f(b)=0$ for some $b>0$ (then, $f(s)<0$ for all $s>b$). Any nonnegative bounded entire solution $u$ of~\eqref{eq:ACN} and~\eqref{localized} necessarily satisfies $0\le u<b$ in $\R\times\R^N$ from the maximum principle. Furthermore, $\max_{\R^N}u(t,\cdot)<b$ for all $t$ negative enough (and then for all $t\in\R$), and it then follows from~\cite{HN} that $u(t,x)$ is a function of $t$ alone and~\eqref{localized} then yields $u\equiv 0$ in $\R\times\R^N$.

Let us now focus on condition~\eqref{hypf2} on $f$. Notice first that it implies that the map~$\phi\mapsto m_\phi$ is constant in $\mathcal E$. Indeed, for any $\phi,\phi'\in\mathcal E$, one has $F<0$ in $(0,m_{\phi'}]$ by~\eqref{hypf2}, while $F(\max_{\R^N}\!\phi)\ge0$ and $f(\max_{\R^N}\!\phi)>0$. Hence, owing to the definition~\eqref{defmphi} of $m_{\phi}$ and the fact that $f(m_{\phi'})=0$, one infers that $m_{\phi'}\le m_{\phi}$. Otherwise $m_\phi<\max_{\R^N}\!\phi<m_{\phi'}$ and thus~$F(\max_{\R^N}\!\phi)<0$, which is a contradiction. Finally $m_\phi=m_{\phi'}$ since $\phi$ and $\phi'$ are arbitrary in~$\mathcal E$.  As a consequence, one also gets that the function $\phi\mapsto M_\phi$ is constant in~$\mathcal{E}$, where the quantity $M_\phi$ has been defined in~\eqref{defMphi}.

\begin{remark}\label{rem13}
From the proof of Theorem~\ref{th1}, it turns out that, in case~(iii) of Theorem~\ref{th1}, one necessarily has $M_\phi<+\infty$ for any (and all) $\phi\in\mathcal E$ and
\be\label{MMphi}
M=M_\phi.
\ee
In particular, if the solution $u$ spreads, it can not converge to an intermediate state smaller than $M$ and the limit state $M$ does not depend on the solution $u$ itself. As a matter of fact, since $F<0$ in $(0,m_\phi]$, $m_\phi<\max_{\R^N}\!\phi<M_\phi$, $F(\max_{\R^N}\!\phi)\ge0$, $f>0$ in $(m_\phi,M_\phi)$ and~$f(M_\phi)=0$ for all $\phi\in\mathcal E$, one then infers that $M=M_\phi$ is the smallest zero of $f$ for which $F(M)>0$, that is,
\be\label{defM}
M=\min\big\{s\ge0: f(s)=0\hbox{ and }F(s)>0\big\}.
\ee
Therefore, $M$ only depends on the function $f$. Notice that this also implies that case~(iii) is ruled out if $f>0$ in $(m_\phi,+\infty)$ for any (and all) $\phi\in\mathcal E$, hence only cases~(i) or~(ii) may occur in this case.
\end{remark}

Together with~\eqref{hypf}, assumption~\eqref{hypf2} plays a key-role in the dichotomy results between the uniformly localized solutions and the solutions converging as $t\to+\infty$ to a positive constant with a positive spreading rate. Without the assumption~\eqref{hypf2}, the solutions $u$ of~\eqref{eq:ACN} may well converge locally uniformly in~$\R^N$ as~$t\to+\infty$ to a steady state $\phi$ such that~$\lim_{|x|\to+\infty}\phi(x)>0$ (such behaviors are known for the solutions of the associated Cauchy problem with localized initial conditions in dimension $N=1$~\cite{MP2,MP3} or with compactly supported initial conditions in dimensions $N\ge1$~\cite{DP}).

The assumption~\eqref{hypf2} is also essential in the proof of formula~\eqref{defxit} saying that, for spreading solutions, the transition between $M-\epsilon$ and $\epsilon$ (for any $0<\epsilon\le M/2$) has bounded width in any radial direction as $t\to+\infty$. This property refers to the notion of transition fronts (here, as $t\to+\infty$) introduced in~\cite{BH}. In other words, the assumption~\eqref{hypf2} prevents the existence of terraces made of stacked propagating fronts between the top value $M$ and the zero state. The existence and attractivity of radial terraces with radial positions $(0<)\ \xi_1(t)<\cdots<\xi_m(t)$ has been proved in~\cite{DM2} under an additional non-degeneracy assumption of all zeroes of $f$ in~$[0,M]$ (see also~\cite{DGM,FM,GM1,MP3} for further results on terraces for homogeneous or spatially periodic equations in dimension $N=1$,~\cite{R1,R2} for the existence of one-dimensional and radially symmetric terraces for gradient multistable systems,~\cite{P5,P6} for the existence of planar terraces for solutions with front-like initial data in~$\R^N$, and~\cite{GR} for the existence of terraces in spatially periodic equations in $\R^N$). The limit values of the ratios of $\xi_i(t)/t$ are also explicit, see~\cite{DM2}. Here, especially thanks to~\eqref{hypf2}, only a single radial layer can exist, and the asymptotic position $\xi(t)$ of that layer at large times is given in terms of the unique speed of a traveling front $\varphi$ connecting $0$ and $M$ for problem~\eqref{eqvarphi}. We point out that $M$ is asymptotically stable from below since $f$ is positive in a left neighborhood of $M$, but $M$ may not be linearly stable, in the sense that $f'(M)$ may vanish. Actually, formula~\eqref{spreading} is proved even if~$f'(M)=0$ (notice in particular that $f$ may not be monotone in a left neighborhood of~$M$)\footnote{On the other hand, if $f$ were assumed to be monotone, namely nonincreasing, in a left neighborhood of~$M$, then it would follow from~\cite{AW,R3} that $\max_{|x|\le c't}|u(t,x)-M|\to0$ and $\max_{|x|\ge c''t}u(t,x)\to0$ as $t\to+\infty$ for every $0\le c'<c<c''$. This in particular yields~\eqref{spreading}, namely $\lim_{t\to+\infty}\xi(t)/t=c$, but this does not show property~\eqref{defxit} on the boundedness of the radial width of the transition between~$0$ and~$M$ at large times.} and, as such, up to our knowledge, the spreading properties~\eqref{defxit}-\eqref{spreading} are new even in dimension $N=1$. The exact position of the layer~$\xi(t)$ and a quantitative estimate on the attractivity of the radial front with speed $c$ are not clear without the assumption~$f'(M)<0$ (see Remark~\ref{remf'M} below for the case $f'(M)<0$). However, if $x_0$ denotes the point with respect to which the considered solution $u$ is radially symmetric and decreasing, and if $a$ is any fixed real number in $(0,M)$, then one knows from~\eqref{spreading} that, for all $t$ large enough, there is a unique $\xi_a(t)\in\R$ such that $u(t,x_0+\xi_a(t)e)=a$ for all unit vectors $e$, and
$$\limsup_{t\to+\infty}|\xi_a(t)-\xi(t)|<+\infty.$$
It is reasonable to conjecture that $u(t,x_0+\xi_a(t)e+x)\to\varphi(x\cdot e+\varphi^{-1}(a))$ as $t\to+\infty$ locally uniformly in $x\in\R^N$ for any unit vector $e$, albeit the proof of this property would require different arguments from the ones used here.

\begin{remark}\label{remf'M}
In alternative~(iii) of Theorem~\ref{th1}, if $M$ is further assumed to be nondegenerate, meaning here that $f'(M)<0$, then there are $x_0\in\R^N$ and $\tau\in\R$ depending on $u$, such that
\be\label{convfront}
\sup_{x\in\R^N}\Big|u(t,x)-\varphi\Big(|x-x_0|-ct+\frac{N-1}{c}\ln t+\tau\Big)\Big|\to0\ \hbox{ as }t\to+\infty.
\ee
In particular, property~\eqref{defxit} then holds with $\xi(t)=ct-((N-1)/c)\ln t$ (say for $t\ge1$) if~$f'(M)<0$. Property~\eqref{convfront} makes the position $\xi(t)$ and the limit profile of $u$ in all directions exactly known up to an~$o(1)$ term as $t\to+\infty$. When $0$ and $M$ are nondegenerate and~$f$ has a single zero in the interval~$(0,M)$, formula~\eqref{convfront} for the solutions converging locally to~$M$ follows from~\cite[Corollary~2]{U} (see also~\cite{R4} for more precise estimates on the position of the front at large times, and~\cite{J} for earlier but less precise estimates). It is easily seen from~\cite{U} that the proof extends to the case when $0$ and $M$ are still nondegenerate and~$f$ has more than one zero in $(0,M)$, since the unique profile $\varphi$ given in~\eqref{eqvarphi}, with unique speed $c>0$, still satisfies $\varphi'<0$ in $\R$ and converges exponentially to $0$ and $M$ at $\pm\infty$.
\end{remark}

Finally, let us comment the assumption~\eqref{localized} on $u$. It is essential in the derivation of~\eqref{conv1} saying that the $\alpha$-limit set of the considered solutions is included in~$\mathcal E$. If, instead of~\eqref{localized}, one only assumes that~$u(t,x)\to 0$  as $|x|\to+\infty$ for each $t\le 0$ (and then, equivalently, for each $t\in\R$), but without any uniformity with respect to $t\le 0$, then the conclusion does not hold in general. For instance, consider any $b>0$ and a function $f$ which is of the bistable type on $[0,b]$, namely there is $a\in(0,b)$ such that
\be\label{cond-f}\left\{\baa{l}
f(0)=f(a)=f(b)=0,\ \ f'(0)<0,\ \ f'(b)<0,\ \ f'(a)>0,\vspace{3pt}\\
f<0\hbox{ in }(0,a),\ \ f>0\hbox{ in }(a,b).\eaa\right.
\ee
The equation~\eqref{eq:ACN} with $f$ satisfying~\eqref{cond-f} was originally proposed in~\cite{AC,NYA} and is accordingly often called the Allen-Cahn equation or the Nagumo equation. It arises in a wide variety of contexts such as phase transition, combustion, ecology and many models of biology. If $f$ satisfies~\eqref{cond-f} and if
\be\label{cond-f2}
\int_0^bf(s)ds>0,
\ee
then there are entire solutions $u:\R\times\R\to(0,b)$ of~\eqref{eq:ACN} in dimension $N=1$ that satisfy $\lim_{x\to\pm\infty}u(t,x)=0$ for each $t\in\R$ and behave as two further and further pulses as $t\to-\infty$, see~\cite{MP1}. Namely, there is a solution $\phi:\R\to(0,a']$ of~\eqref{steady}, where $a'\in(a,b)$ is such that~$F(a')=0$ and $\phi$ is the unique, up to shifts, solution of~\eqref{steady} ranging in $[0,b]$. The solutions $u$ constructed in~\cite{MP1} are such that
$$u(t,x)-\phi(-\xi(t))-\phi(\xi(t))\to0\ \hbox{ as }t\to-\infty\ \hbox{ uniformly in }x\in\R,$$
with $\lim_{t\to-\infty}\xi(t)=+\infty$. Therefore, these solutions $u$ do not satisfy~\eqref{conv1}. On the other hand, if $f$ satisfies~\eqref{cond-f} and
\be\label{cond-f3}
\int_0^bf(s)ds<0,
\ee
then there are entire solutions $u:\R\times\R\to(0,b)$ of~\eqref{eq:ACN} that satisfy $\lim_{x\to\pm\infty}u(t,x)=0$ for each $t\in\R$ and behave as two far fronts as $t\to-\infty$, see~\cite{GM2}. Namely, under~\eqref{cond-f} and~\eqref{cond-f3}, equation~\eqref{eq:ACN} admits a traveling front $\varphi(x-ct)$ with $c<0$ and such that~$\varphi:\R\to(0,b)$ is decreasing with $\varphi(-\infty)=b$ and $\varphi(+\infty)=0$, that is, $\varphi$ obeys~\eqref{eqvarphi} with $M=b$. The solutions $u$ constructed in~\cite{GM2} satisfy
$$u(t,x)-\varphi(x-ct)-\varphi(-x-ct)+b\to0\ \hbox{ as }t\to-\infty\hbox{ uniformly in }x\in\R,$$
and also $\|u(t,\cdot)\|_{L^\infty(\R)}\to0$ as $t\to+\infty$. Since in this case there is $b'\in(b,+\infty)$ such that~$F<0$ in $(0,b']$, equation~\eqref{steady} does not admit any solution $\phi:\R\to[0,b']$, and the solutions~$u$ then do not satisfy~\eqref{conv1}.

\begin{remark}\label{rem15}{\rm In Theorem~\ref{th1}, the assumption~\eqref{localized} can nevertheless be relaxed, still keeping the uniformity with respect to $t\le0$. More precisely, let any $\eta>0$ be such that $f<0$ in~$(0,\eta]$. Notice that such a real number $\eta$ exists by~\eqref{hypf}. It then turns out that Theorem~\ref{th1} holds if~\eqref{localized} is replaced by the assumption
\be\label{localized2}
\limsup_{|x|\to+\infty}\Big(\sup_{t\le0}u(t,x)\Big)\le\eta,
\ee
or $\limsup_{|x|\to+\infty}\big(\sup_{t\le t_0}u(t,x)\big)\le\eta$ for some $t_0\in\R$. As a matter of fact,~\eqref{localized2} implies (and is then equivalent to)~\eqref{localized}. Indeed, assume by way of contradiction that~\eqref{localized2} holds, but not~\eqref{localized}. Then there is a sequence $(t_n,x_n)_{n\in\N}$ in~$(-\infty,0]\times\R^N$ such that
$$\lim_{n\to+\infty}|x_n|=+\infty\ \hbox{ and }\ 0<\liminf_{n\to+\infty}u(t_n,x_n)\le\limsup_{n\to+\infty}u(t_n,x_n)\le\eta.$$
Furthermore, it follows from standard parabolic estimates that, up to extraction of a subsequence, the functions $u_n:(t,x)\mapsto u_n(t,x)=u(t+t_n,x+x_n)$ converge in $C^{1,2}_{t,x}$ locally in~$\R\times\R^N$ to a nonnegative bounded solution~$u_\infty$ of~\eqref{eq:ACN} such that $u_\infty(0,0)>0$ and~$0\le u_\infty\le\eta$ in~$(-\infty,0]\times\R^N$. Therefore, the maximum principle yields $u_\infty(t,x)\le\zeta(t+t_0)$ for all~$t_0>0$ and $(t,x)\in[-t_0,+\infty)\times\R^N$, where $\zeta$ obeys $\zeta(0)=\eta$ and $\zeta'(t)=f(\zeta(t))$ for all~$t\ge0$. In particular, $0<u_\infty(0,0)\le\zeta(t_0)$ for all $t_0>0$. But $\zeta(+\infty)=0$ since $f<0$ in $(0,\eta]$ and~$f(0)=0$, leading to a contradiction. Therefore,~\eqref{localized} could equivalently be replaced by~\eqref{localized2} in Theo\-rem~\ref{th1}, but we preferred to state Theo\-rem~\ref{th1} with~\eqref{localized} since this assumption is simpler to write and does not involve the  additional introduction of a quantity~$\eta$.}
\end{remark}


\subsection{Existence of monotone heteroclinic connections}

Let us now discuss the existence of heteroclinic connections between a steady state $\phi$ of~\eqref{steady} and the constant states~$0$ or~$M$.  The results described in the following paragraphs are quite standard and inspired by similar ones in bounded domains~\cite{M}. See also \cite[Theorem 1.2]{FMN}. We just sketch here the main steps of the proof of the existence of heteroclininc connections for the sake of completeness. So, consider any solution~$\phi$ of~\eqref{steady}. Since $\phi$ decays exponentially to $0$ at infinity (see also Section~\ref{sec41}), so do its first- and second-order partial derivatives, from standard elliptic estimates. In particular,~$\phi\in H^1(\R^N)$ and~$\phi_{x_i}\in H^1(\R^N)$ for each $1\le i\le N$. Since each first-order partial derivative $\phi_{x_i}$ changes sign and satisfies the equation $\Delta\phi_{x_i}+f'(\phi)\phi_{x_i}=0$ in $\R^N$, it follows that $\phi$ is a strictly unstable solution of~\eqref{steady}, in the sense that, for all~$R>0$ large enough, the principal eigenvalue $\lambda_R$ of the operator $-\Delta-f'(\phi)$ in the open Euclidean ball $B_R$ with center $0$ and radius $R$, with Dirichlet boundary condition on $\partial B_R$, is such that~$\lambda_R<0$. Let~$\varphi_R\in C^2(\overline{B_R})$ be a principal eigenfunction associated to this operator, namely
\be\label{eigenfunction}
-\Delta\varphi_R-f'(\phi)\varphi_R=\lambda_R\varphi_R\hbox{ in }B_R,\ \ \varphi_R>0\hbox{ in }B_R,\ \hbox{ and }\ \varphi_R=0\hbox{ on }\partial B_R.
\ee
Fix any $R>0$ large enough such that $\lambda_R<0$. There exists then $\epsilon^*>0$ such that, for all~$\epsilon\in(0,\epsilon^*)$, the $C^2(\overline{B_R})$ function $\phi_{R,\epsilon}:=\phi-\epsilon\varphi_R$ satisfies $0<\phi_{R,\epsilon}\le\phi$ in $\overline{B_R}$, $\phi_{R,\epsilon}=\phi$ on~$\partial B_R$ and
$$\Delta\phi_{R,\epsilon}+f(\phi_{R,\epsilon})<0\ \hbox{ in }B_R.$$
Denote $\widetilde{\phi}_{R,\epsilon}(x)=\phi_{R,\epsilon}(x)$ if $x\in\overline{B_R}$ and $\widetilde{\phi}_{R,\epsilon}(x)=\phi(x)$ if $x\in\R^N\!\setminus\!\overline{B_R}$. Thus, for any $\epsilon\in(0,\epsilon^*)$, the bounded and uniformly continuous function $\widetilde{\phi}_{R,\epsilon}$ is a generalized strict supersolution of~\eqref{steady}, and the solution $u^\epsilon$ of the Cauchy problem
\be\label{defueps}\left\{\baa{rcll}
u^\epsilon_t & \!\!=\!\! & \Delta u^\epsilon+f(u^\epsilon) & \hbox{in }(0,+\infty)\times\R^N,\vspace{3pt}\\
u^\epsilon(0,\cdot) & \!\!=\!\! & \widetilde{\phi}_{R,\epsilon} & \hbox{in }\R^N,\eaa\right.
\ee
is strictly decreasing in $t$. Furthermore, it is truly globally defined (in $[0,+\infty)\times\R^N$) and satisfies $0<u^\epsilon<\tilde{\phi}_{R,\epsilon}\le\phi$ in $(0,+\infty)\times\R^N$ from the parabolic maximum principle. The monotonicity in time and standard parabolic estimates yield the existence of a $C^2(\R^N)$ solution~$\phi'$ of $\Delta\phi'+f(\phi')=0$ in $\R^N$ such that $0\le\phi'<\phi$ in $\R^N$. Since by~\cite{BJP} any two solutions of~\eqref{steady} can not be ordered, it follows from the strong elliptic maximum principle that~$\phi'\equiv0$ in $\R^N$. Without loss of generality, one can assume that $\epsilon^*\varphi_R(0)<\phi(0)/2$. Hence, for any $\epsilon\in(0,\epsilon^*)$, one has $u^\epsilon(0,0)>\phi(0)/2$, and there is a time $t^\epsilon>0$ such that
$$u^\epsilon(t^\epsilon,0)=\frac{\phi(0)}{2}.$$
Notice that $t^\epsilon$ is unique since $u^\epsilon$ is continuous and decreasing in $t$. Furthermore, from the continuous dependence of the solutions of the Cauchy problem associated to~\eqref{eq:ACN} with respect to the initial data, and since $\phi$ is a steady state, one infers that $t^\epsilon\to+\infty$ as $\epsilon\to0$. Consider now a decreasing sequence $(\epsilon_n)_{n\in\N}$ in $(0,\epsilon^*)$ and converging to $0$, and define
$$u_n(t,x)=u^{\epsilon_n}(t+t^{\epsilon_n},x)\ \hbox{ for }(t,x)\in[-t^{\epsilon_n},+\infty)\times\R^N.$$
From the previous observations and standard parabolic estimates, the functions $u_n$ converge up to extraction of a subsequence in $C^{1,2}_{t,x}$ locally in $\R\times\R^N$ to a solution $u_{\infty}$ of~\eqref{eq:ACN}, such that $0\le u_\infty\le\phi$ and $(u_\infty)_t\le0$ in $\R\times\R^N$, together with $u_\infty(0,0)=\phi(0)/2$. Hence,
$$0<u_\infty<\phi\ \hbox{ in }\R\times\R^N$$
from the strong maximum principle, and there are two steady states $\phi_\pm\in C^2(\R^N)$ such that~$u_\infty(t,\cdot)\to\phi_\pm$ as $t\to\pm\infty$ locally uniformly in~$\R^N$ (and then uniformly since $0<u_\infty<\phi$ in~$\R\times\R^N$ and $\phi(x)\to0$ as $|x|\to+\infty$), with $0\le\phi_+\le\phi_-\le\phi$ and $\phi_+(0)\le\phi(0)/2\le\phi_-(0)$. The strong maximum principle and the non-existence of ordered solutions of~\eqref{steady} imply that~$\phi_-\equiv\phi$ and $\phi_+\equiv0$ in $\R^N$. Lastly, $(u_\infty)_t<0$ from the strong parabolic maximum principle applied to this function. In other words, the solution $u_\infty$ of~\eqref{eq:ACN} is a time-decreasing heteroclinic connection from $\phi$ to $0$.

Similarly, the function $\phi+\epsilon\varphi_R\in C^2(\overline{B_R})$ satisfies $\Delta(\phi+\epsilon\varphi_R)+f(\phi+\epsilon\varphi_R)>0$ in~$B_R$ for all $\epsilon>0$ small enough, with $R>0$ fixed large enough as above. Assume here that~$M_\phi$, defined in~\eqref{defMphi}, is a real number. Notice that the arguments of~\cite{BJP} based on the maximum principle and the sliding method imply that, for any classical positive solution $\phi'$ of $\Delta\phi'+f(\phi')=0$ in $\R^N$ such that $\phi'\ge\phi$, one has either $\phi'\equiv\phi$ in $\R^N$, or $\phi'\ge\max_{\R^N}\phi$ and then $\phi'\ge M_\phi$ in~$\R^N$ (remember that $f>0$ in $[\max_{\R^N}\!\phi,M_\phi)$). Therefore, it follows that, for all $\epsilon>0$ small enough, the solutions~$u^\epsilon$ of~\eqref{defueps}, with this time
\be\label{deftildephi}
\widetilde{\phi}_{R,\epsilon}=\left\{\baa{ll}
\phi+\epsilon\varphi_R & \hbox{in }\overline{B_R},\vspace{3pt}\\
\phi & \hbox{in }\R^N\!\setminus\!\overline{B_R},\eaa\right.
\ee
are increasing in time in $(0,+\infty)\times\R^N$, and satisfy $u^\epsilon(t,\cdot)\to M_\phi$ as $t\to+\infty$ locally uniformly in $\R^N$. Furthermore, since $\phi$ is radially symme\-tric (with respect to, say, the origin without loss of generality) and decreasing in $|x|$, with $\Delta\phi(0)=-f(\phi(0))=-f(\max_{\R^N}\!\phi)<0$, and since the principal eigenfunction $\varphi_R$ of~\eqref{eigenfunction} is itself radially symmetric by uniqueness, it follows that the functions $\widetilde{\phi}_{R,\epsilon}$ given in~\eqref{deftildephi} are also radially symmetric and decreasing in~$|x|$, for all~$\epsilon>0$ small enough. So are the functions $u^\epsilon(t,\cdot)$ for all $t>0$. With the same arguments as in the previous paragraph, by defining a time $t^\epsilon>0$ such that
$$u^\epsilon(t^\epsilon,0)=\frac{\phi(0)+M_\phi}{2},$$
one infers the existence of a time-increasing heteroclinic connection between~$\phi$ and~$M_\phi$ (the convergence to $M_\phi$ being this time only locally uniform in $\R^N$ as $t\to+\infty$).

For bistable functions $f$ of the type~\eqref{cond-f}-\eqref{cond-f2}, we also refer to~\cite{HR2,MN} for the existence of other time-increasing heteroclinic connections $u(t,x_1,x_2)=U(x_1,x_2-\gamma t)$ of~\eqref{eq:ACN}, for any $\gamma>0$ large enough, between the extended one-dimensional solution $\phi(x_1,x_2)=\phi(x_1)$ of~\eqref{steady} and the constant $b=M_\phi$, in the sense that $u(t,x_1,x_2)\to\phi(x_1)$ as $t\to-\infty$ uniformly in $x_1$ and locally uniformly in $x_2$, and $u(t,x_1,x_2)\to b$ as $t\to+\infty$ locally uniformly in $(x_1,x_2)$. These connections are however not localized, that is, they do not satisfy~\eqref{localized}: as a matter of fact, one has $\limsup_{|(x_1,x_2)|\to+\infty}u(t,x_1,x_2)=b>0$ for each $t\in\R$.

Assume now that $M_\phi=+\infty$. In that case, the solutions $u^\epsilon$ of the previous paragraphs can not stay bounded and therefore blow up in finite or infinite time, according to the behavior of $f(s)$ as $s\to+\infty$. As above, without loss of generality, for all $\epsilon>0$ small enough, the functions $u^\epsilon(t,\cdot)$ are radially symmetric and decreasing in $|x|$ for all $t$ in their interval $(0,T^\epsilon)$ of existence, hence~$u^\epsilon(t,0)\to+\infty$ as $t\to T^\epsilon$. For all $\epsilon>0$ small enough, there is then a time $t^\epsilon\in(0,T^\epsilon)$ such that
$$u^\epsilon(t^\epsilon,0)=\phi(0)+1,$$
and $t^\epsilon\to+\infty$ as $\epsilon\to0$. Therefore, using again the strong parabolic maximum principle, there is a time-increasing solution $u_\infty$ of~\eqref{eq:ACN}, defined now in $(-\infty,T)\times\R^N$ with $T\in(0,+\infty]$, such that $u_\infty(0,0)=\phi(0)+1$,~$u_\infty$ is radially symmetric and decreasing with respect to $|x|$, $u_\infty(t,\cdot)\to\phi$ uniformly in $\R^N$ as $t\to-\infty$, and $u_\infty(t,0)\to+\infty$ as $t\to T$. In other words, $u_\infty$ blows up at time $T$ (which may be finite or infinite, according to the function $f$). 

Lastly, we point out that, since any two solutions of~\eqref{steady} can not be ordered~\cite{BJP}, it follows that~\eqref{eq:ACN} can not have any time-monotone heteroclinic connection between two different solutions~$\phi_\pm$ of~\eqref{steady}. However, the existence of non-time-monotone heteroclinic connections is not a priori ruled out.

\begin{remark}
For any positive bounded solution $u$ of~\eqref{eq:ACN} satisfying~\eqref{localized}, the action
$$E[u(t,\cdot)]=\int_{\R^N}\Big(\frac{|\nabla u(t,x)|^2}{2}-F(u(t,x)\Big)\,dx$$
is well defined and it is a Lyapunov functional, that is, $t\mapsto E[u(t,\cdot)]$ is non-increasing in~$\R$ and even decreasing unless $u$ does not depend on $t$. We refer to Section~\ref{sec41} for more details. Notice that $E[\phi]>0=E[0]$ for any solution $\phi$ of~\eqref{steady} (this can be viewed as a consequence of the aforementioned existence of heteroclinic connections between~$\phi$ and~$0$). If $M_\phi$ is a real number and $u$ is a heteroclinic connection between $\phi$ and $M_\phi$, or more generally speaking in case~(iii) of Theorem~\ref{th1}, then $E[u(t,\cdot)]\to-\infty$ as $t\to+\infty$. If $u$ is a heteroclinic connection between two different solutions $\phi_\pm$ of~\eqref{steady} in the sense that $\|u(t,\cdot)-\phi_\pm\|_{L^\infty(\R^N)}\to0$ as~$t\to\pm\infty$ (this is a particular case of alternative~(ii) of Theorem~\ref{th1}), then
$$E[\phi_-]>E[\phi_+].$$
Lastly, if a positive bounded solution $u$ of~\eqref{eq:ACN} satisfying~\eqref{localized} does not converge to a single solution of~\eqref{steady} as $t\to-\infty$, in the sense that there are at least two different solutions $\phi$ and $\phi'$ of~\eqref{steady} such that
$$\|u(t_n,\cdot)-\phi\|_{L^\infty(\R^N)}\to0\ \hbox{ and }\ \|u(t'_n,\cdot)-\phi'\|_{L^\infty(\R^N)}\to0\ \hbox{ as }\ n\to+\infty$$
with $\lim_{n\to+\infty}t_n=\lim_{n\to+\infty}t'_n=-\infty$, then $E[\phi]=E[\phi']$. Notice that in that situation,~$\phi$ and~$\phi'$ are necessarily radially symmetric with respect to the same origin, since so is $u$, and by connectedness of the $\alpha$-limit set of $u$ there is then a continuum of such limit steady states as~$t\to-\infty$ in the $\alpha$-limit set of $u$, all having the same Lagrangian (we also refer to the discussion before Corollary~\ref{cor3} below). As a consequence, if the Lagrangian $E$ is one-to-one of the set the solutions of~\eqref{steady} which are symmetric with respect to the same point, then there is a single~$\phi\in\mathcal E$ such that~$\|u(t,\cdot)-\phi\|_{L^\infty(\R^N)}\to0$ as $t\to-\infty$.
\end{remark}


\section{Some corollaries and particular cases}\label{sec:2}
\setcounter{equation}{0}

In this section, we list some corollaries of Theorem~\ref{th1} which correspond to further assumptions or to some special cases. In particular, the conclusion~\eqref{conv1} will be made more precise under further assumptions.


\subsection{Dimension $N=1$}

The first corollary is concerned with the dimension $N=1$. In this case, the solutions of~\eqref{steady} are unique, up to shifts. Indeed, for any such solution~$\phi$, it follows from~\eqref{profile} and~\eqref{eqradial} that $F<0$ in $(0,\max_\R\phi)=(0,\Phi(0))$ and $\Phi'(r)=-\sqrt{-2F(\Phi(r))}$ for all $r\ge0$, hence the radial profile $\Phi$ is unique from the Cauchy-Lipshitz theorem. Furthermore,~$F(\Phi(0))=0$ and $f(\Phi(0))>0$. One can also infer from~\eqref{eqradial} that, if there is~$\beta\in(0,+\infty)$ such that $F<0$ in $(0,\beta)$ with $F(\beta)=0$ and $f(\beta)>0$ (notice that the hypotheses $F<0$ in~$(0,\beta)$ and $F(\beta)=0$ imply that $f(\beta)\ge0$), then~\eqref{steady} has a (unique up to shifts) solution.

\begin{corollary}\label{cor1}
Assume that $N=1$, that $f$ satisfies~\eqref{hypf} and that there is $\beta\in(0,+\infty)$ such that $F<0$ in $(0,\beta)$ with $F(\beta)=0$ and $f(\beta)>0$. Then $\mathcal E\neq\emptyset$ and, for any positive bounded solution~$u$ of~\eqref{eq:ACN} satisfying~\eqref{localized}, there is $\phi\in\mathcal E$ such that $\|u(t,\cdot)-\phi\|_{L^\infty(\R)}\to0$ as $t\to-\infty$. Furthermore, either $\|u(t,\cdot)\|_{L^\infty(\R)}\to0$ as $t\to+\infty$, or $u(t,x)\equiv\phi(x)$ in $\R\times\R$, or else the alternative~{\rm{(iii)}} of Theorem~$\ref{th1}$ holds.
\end{corollary}

Corollary~\ref{cor1} easily follows from Theorem~\ref{th1}, the previous observations and the fact that the solutions $u$ are necessarily even in $x$ with respect to some real number. The fact that case~(ii) reduces to $u\equiv\phi$ is a consequence of the existence of a Lyapunov functional and the uniqueness of the solutions of~\eqref{steady} up to shifts. We refer to Section~\ref{sec5} for more details.


\subsection{Non-existence of positive bounded solutions}

In dimension $N=1$, with~\eqref{hypf}, the existence of a smallest positive root $\beta$ of $F$ with $f(\beta)>0$ is a necessary and sufficient condition for the existence of a solution of~\eqref{steady}. In dimension $N\ge 2$, any solution~$\phi$ of~\eqref{steady} satisfies $F(\max_{\R^N}\phi)>0$. Therefore, the next result immediately follows from Theorem~\ref{th1} and the strong maximum principle.

\begin{corollary}\label{cor2}
Assume that $f$ satisfies~\eqref{hypf}. In dimension $N=1$, if $F<0$ in $(0,+\infty)$ or if $F<0$ in $(0,\beta)$ with $F(\beta)=0$ and $f(\beta)=0$, then the only nonnegative bounded solution~$u$ of~\eqref{eq:ACN} satisfying~\eqref{localized} is the trivial solution $u\equiv 0$ in $\R\times\R$. In dimension $N\ge 2$, if $F\le 0$ in $[0,+\infty)$, then the same conclusion holds.
\end{corollary}

The assumptions of Corollary~\ref{cor2} are simple conditions ruling out the existence of solutions to~\eqref{steady}. These assumptions are however not optimal in dimensions $N\ge2$. For instance, if $N\ge3$, for any positive real numbers $\gamma$ and $\delta$ and for any $p\ge(N+2)/(N-2)$, the equation~\eqref{steady} with
\be\label{deff1}
f(s)=-\gamma s+\delta s^p
\ee
does not admit any solution~\cite{KP,P}. The same property holds with
\be\label{deff2}
f(s)=-\gamma s-\delta s^p+\eta s^q
\ee
with $N\ge3$, $\gamma>0$, $\delta>0$, $\eta>0$, and $1<p\le(N+2)/(N-2)\le q$ or $(N+2)/(N-2)<p<q$, see~\cite{KP,NS,P}. In these two examples, Theorem~\ref{th1} implies that the only nonnegative bounded solution $u$ of~\eqref{eq:ACN} satisfying~\eqref{localized} is then the trivial solution $u\equiv 0$ in $\R\times\R^N$.

On the other hand, much work has been devoted to the existence of solutions to~\eqref{steady} for some classes of functions $f$ satisfying~\eqref{hypf}, see the book~\cite{KP}. For instance, if $N\ge3$ and $\sup_{[0,+\infty)}F>0$ together with $\max(f(s),0)=o(s^{(N+2)/(N-2)})$ as $s\to+\infty$, then~\eqref{steady} admits solutions, see~\cite{BL1,BL2,S}. The existence holds in dimension $N=2$ if for instance $f$ satisfies~\eqref{cond-f}-\eqref{cond-f2}, see~\cite{BLP}, or if $\sup_{[0,+\infty)}F>0$ and $\max(f(s),0)=o(e^{\alpha\,s^2})$ as $s\to+\infty$ for all $\alpha>0$, see~\cite{BGK}.


\subsection{Discreteness or uniqueness up to shifts of the localized steady states}

In Theorem~\ref{th1}, property~\eqref{conv1} says that the solution $u$ is close to the family of steady states of~\eqref{steady} as $t\to-\infty$, that is, the $\alpha$-limit set of $u$ (with respect to the uniform convergence in~$\R^N$) consists of solutions of~\eqref{steady}, which turn out to be all symmetric with respect to a same point in $\R^N$, since so is~$u$. Any two different solutions of~\eqref{steady} can not be ordered by~\cite{BJP}, but it is not clear in general to know whether $u$ emanates from a single steady state or from a continuum of them (the possible existence of continua of solutions of~\eqref{steady} which are symmetric with respect to the same point is a difficult issue in general dimensions~$N\ge2$). However, under some further assumptions on the set of steady states, combined with the connectedness of the $\alpha$-limit sets of the solutions, one can be sure that a single state is selected.

\begin{corollary}\label{cor3}
Assume that $f$ satisfies~\eqref{hypf} and~\eqref{hypf2}, and that the set of solutions of~\eqref{steady} which are radially symmetric with respect to the origin is discrete.\footnote{This means that, for any radially symmetric solution $\phi$ of~\eqref{steady}, there is $\epsilon>0$ such that $\|\psi-\phi\|_{L^\infty(\R^N)}\ge\epsilon$ for every radially symmetric solution $\psi$ of~\eqref{steady}.} Let $u$ be a positive bounded solution of~\eqref{eq:ACN} satisfying~\eqref{localized}. Then there is $\phi\in\mathcal E$ such that $\|u(t,\cdot)-\phi\|_{L^\infty(\R^N)}\to0$ as~$t\to-\infty$. Furthermore,
\begin{enumerate}
\item either $u(t,\cdot)\to0$ uniformly in $\R^N$ as $t\to+\infty$,
\item or there is $\phi'\in\mathcal E$ such that $u(t,\cdot)\to\phi'$ uniformly in $\R^N$ as $t\to+\infty$, and, if $\phi=\phi'$, then $u(t,x)\equiv\phi(x)\equiv\phi'(x)$ in $\R\times\R^N$, 
\item or else the alternative~{\rm{(iii)}} of Theorem~$\ref{th1}$ holds.
\end{enumerate}
\end{corollary}

In the particular case when the solutions $\phi$ of~\eqref{steady} are unique up to shifts, then $F\le0$ in~$[0,m_\phi]$ since otherwise, by applying the results of~\cite{BL2,BLP} to the function $f$ extended by~$0$ in~$(m_\phi,+\infty)$, there would exist other solutions $\phi'$ of~\eqref{steady} such that
$$\max_{\R^N}\phi'<m_\phi<\max_{\R^N}\phi.$$
Hence, condition~\eqref{hypf2} is necessarily fulfilled if the solutions of~\eqref{steady} are unique up to shifts. Therefore, since the bounded positive solutions $u$ of~\eqref{eq:ACN} satisfying~\eqref{localized} are necessarily radially symmetric and decreasing with respect to a single point in $\R^N$, the following corollary immediately holds.

\begin{corollary}\label{cor4}
Assume that $f$ satisfies~\eqref{hypf} and that the solutions of~\eqref{steady} exist and are unique up to shifts. Let $u$ be a positive bounded solution of~\eqref{eq:ACN} satisfying~\eqref{localized}. Then there is $\phi\in\mathcal E$ such that $\|u(t,\cdot)-\phi\|_{L^\infty(\R^N)}\to0$ as $t\to-\infty$. Moreover, either $\|u(t,\cdot)\|_{L^\infty(\R^N)}\to0$ as $t\to+\infty$, or $u(t,x)\equiv\phi(x)$ in $\R\times\R^N$, or else the alternative~{\rm{(iii)}} of Theorem~$\ref{th1}$ holds.
\end{corollary}

The existence and uniqueness up to shifts of the solutions of~\eqref{steady} is known for some classes of functions $f$. For instance, if $f$ satisfies~\eqref{hypf} and if there are $0<a<a'<b\le+\infty$ such that $f<0$ in $(0,a)$, $f>0$ in $(a,b)$, $F(a')=0$, $f\le0$ in $[b,+\infty)$ and
\be\label{fmono}
s\mapsto\frac{f(s)}{s-a'}\hbox{ is nonincreasing in }(a',b),
\ee
then there exists a unique up to shifts solution $\phi$ of~\eqref{steady}, see~\cite{FLS,PS} (in this case, $m_\phi=a$ and $M_\phi=b$). Notice that the monotonicity condition~\eqref{fmono} is especially fulfilled if $f$ is nonincreasing in $[a',b)$. The condition~\eqref{fmono} is not optimal for the uniqueness, since there are bistable functions $f$ satisfying the above conditions but~\eqref{fmono} for which the uniqueness up to shifts holds for~\eqref{steady} in any dimension $N\ge1$ (see especially the cubic functions~$f$ of the type~\eqref{example-f} below used in Corollary~\ref{cor7}). The uniqueness up to shifts of the solutions of~\eqref{steady} also holds if $f$ satisfies~\eqref{hypf} and if there is $a\in(0,+\infty)$ such that $f\le0$ in $[0,a]$,~$f>0$ in~$(a,+\infty)$ and $s\mapsto sf'(s)/f(s)$ is nonincreasing in $(a,+\infty)$, see~\cite{ASW,ST}. An example is the function $f$ given in~\eqref{deff1}, namely
$$f(s)=-\gamma s+\delta s^p,$$
with $\gamma>0$, $\delta>0$ and $p>1$. As a matter of fact, for that function, the existence and uniqueness up to shifts of a solution $\phi$ of~\eqref{steady} holds if and only if $N\le2$, or $N\ge3$ and~$1<p<(N+2)/(N-2)$, see also~\cite{BL1,BL2,BLP,CL,C,CEF,KP,K,Mc,MS,P,Y}. In this case, one has $m_\phi=(\gamma/\delta)^{1/(p-1)}$ and $M_\phi=+\infty$, and it follows from Corollary~\ref{cor4} that any positive bounded solution of~\eqref{eq:ACN} satisfying~\eqref{localized} is either independent of $t$ or converges to~$0$ uniformly in $\R^N$ as $t\to+\infty$. Another important example is that of functions $f$ of the type~\eqref{deff2}, namely
$$f(s)=-\gamma s-\delta s^p+\eta s^q,$$
with $\gamma>0$, $\delta>0$, $\eta>0$, $p\neq q$, and $\min(p,q)>1$. The uniqueness up to shifts of the solutions~$\phi$ of~\eqref{steady} holds in that case, and the existence holds if and only if $N\le 2$, or~$p<q<(N+2)/(N-2)$ with $N\ge3$ (in this case, $M_\phi=+\infty$ and the bounded solutions of~\eqref{eq:ACN} satisfying~\eqref{localized} are either independent of $t$ or converge to $0$ uniformly in~$\R^N$ as~$t\to+\infty$), or $p>q$ with $N\ge3$ and $\beta$ is small enough (in this case, $M_\phi<+\infty$), see~\cite{ST}.

On the other hand, without~\eqref{fmono} or the aforementioned conditions listed in the previous paragraph, some examples of non-uniqueness up to shifts in $\R^N$ with $N\ge2$ are known, for functions $f$ of the bistable type~\eqref{cond-f} with $f<0$ in $(b,+\infty)$, see~\cite{PS}, or for some functions $f$ having one single positive zero, see~\cite{P0} (notice that conditions~\eqref{hypf} and~\eqref{hypf2} are automatically fulfilled for the functions considered in~\cite{PS,P0}). In~\cite{DDG}, for functions~$f$ of the type $f(s)=-s+s^p+\lambda s^q$ with $\lambda>0$ large, $1<q<3$ and $p<5$ close to $5$ (conditions~\eqref{hypf} and~\eqref{hypf2} then hold), it was shown that~\eqref{steady} in dimension $N=3$ admits at least three solutions which are radially symmetric with respect to the origin. Furthermore, it is reasonable to conjecture from the proof given in~\cite{DDG} that the set of all such solutions is discrete, in which case Corollary~\ref{cor3} can be applied.


\subsection{Bistable and cubic functions $f$}

We complete this section by considering the class of bistable functions $f$ satisfying~\eqref{cond-f} for some $0<a<b$, namely
\be\label{cond-fbis}\left\{\baa{l}
f(0)=f(a)=f(b)=0,\ \ f'(0)<0,\ \ f'(b)<0,\ \ f'(a)>0,\vspace{3pt}\\
f<0\hbox{ in }(0,a),\ \ f>0\hbox{ in }(a,b),\eaa\right.
\ee
together with
\be\label{cond-f4}
f<0\ \hbox{ in }(b,+\infty).
\ee

On the one hand, if
\be\label{cond-f5}
\int_0^bf(s)\,ds\le0,
\ee
then $F<0$ in $(0,b)\cup(b,+\infty)$, hence Corollary~\ref{cor2} immediately yields the following result.

\begin{corollary}\label{cor5}
If $f$ satisfies~\eqref{cond-fbis}-\eqref{cond-f5}, then, in any dimension $N\ge1$, the only nonnegative bounded solution $u$ of~\eqref{eq:ACN} satisfying~\eqref{localized} is the trivial solution $u\equiv 0$ in $\R\times\R^N$.
\end{corollary}

On the other hand, if
\be\label{cond-f6}
\int_0^bf(s)\,ds>0,
\ee
then there are solutions $\phi$ of~\eqref{steady} \cite{BL2,BLP}, and
$$a=m_\phi<\max_{\R^N}\phi<M_\phi=b.$$
It is also well known~\cite{AW,FM} that there is a unique $c\in\R$, which is positive, and a unique up to shift function $\varphi\in C^2(\R)$ such that
\be\label{eqvarphi2}
\varphi''+c\varphi'+f(\varphi)=0\hbox{ in }\R,\ \ \varphi'<0\hbox{ in }\R,\ \ \varphi(-\infty)=b,\ \ \varphi(+\infty)=0.
\ee

An immediate corollary of Corollary~\ref{cor4}, Theorem~\ref{th1} and property~\eqref{convfront} in Remark~\ref{remf'M} is the following result.

\begin{corollary}\label{cor6}
Assume that $f$ satisfies~\eqref{cond-fbis} and \eqref{cond-f6} and that the solutions of~\eqref{steady}  are unique up to shifts. Let $u$ be a positive bounded solution of~\eqref{eq:ACN} satisfying~\eqref{localized}. Then there is $\phi\in\mathcal E$ such that $\|u(t,\cdot)-\phi\|_{L^\infty(\R^N)}\to0$ as $t\to-\infty$. Moreover, either $\|u(t,\cdot)\|_{L^\infty(\R^N)}\to0$ as $t\to+\infty$, or $u(t,x)\equiv\phi(x)$ in $\R\times\R^N$, or else there are $x_0\in\R^N$ and $\tau\in\R$ such that
$$\sup_{x\in\R^N}\Big|u(t,x)-\varphi\Big(|x-x_0|-ct+\frac{N-1}{c}\ln t+\tau\Big)\Big|\to0\ \hbox{ as }t\to+\infty,$$
where $c>0$ and $\varphi\in C^2(\R^N)$ are given in~\eqref{eqvarphi2}.
\end{corollary}

Consider finally an important example of functions $f$ satisfying~\eqref{cond-f}, namely the cubic nonlinearities
\be\label{example-f}
f(s)=s(b-s)(s-a)
\ee
with $0<a<b$. Notice that~\eqref{cond-f6} is fulfilled if and only if $a<b/2$. Furthermore, in that case, the solutions of~\eqref{steady} exist and are unique up to shifts, by~\cite{BL2,BLP,ST}. As a consequence, the following corollary holds.

\begin{corollary}\label{cor7}
Let $0<a<b$ and $f$ be of the type~\eqref{example-f}. If $a\ge b/2$, then, in any dimension~$N\ge1$, the only nonnegative bounded solution $u$ of~\eqref{eq:ACN} satisfying~\eqref{localized} is the trivial solution $u\equiv 0$ in $\R\times\R^N$. If $a<b/2$, then~\eqref{steady} admits solutions and, for any positive bounded solution $u$ of~\eqref{eq:ACN} satisfying~\eqref{localized}, there is $\phi\in\mathcal E$ such that $\|u(t,\cdot)-\phi\|_{L^\infty(\R^N)}\to0$ as $t\to-\infty$, and either $\|u(t,\cdot)\|_{L^\infty(\R^N)}\to0$ as $t\to+\infty$, or $u(t,x)\equiv\phi(x)$ in $\R\times\R^N$, or else there are $x_0\in\R^N$ and $\tau\in\R$ such that
$$\sup_{x\in\R^N}\Big|u(t,x)-\varphi\Big(|x-x_0|-ct+\frac{N-1}{c}\ln t+\tau\Big)\Big|\to0\ \hbox{ as }t\to+\infty,$$
where $c>0$ and $\varphi\in C^2(\R^N)$ are given in~\eqref{eqvarphi2}.
\end{corollary}

\noindent{\bf{Outline of the paper.}} Section~\ref{sec3} is concerned with some preliminary results on the existence of planar traveling fronts connecting $0$ and $M$ (these fronts are used in alternative~(iii) of Theorem~\ref{th1}) and on the existence of steady states in large balls under some assumptions on the nonlinearity. Section~\ref{sec4} is devoted to the proof of Theorem~\ref{th1} and Section~\ref{sec5} to the proof of Corollaries~\ref{cor1} and~\ref{cor3} (the other corollaries follow from the other results, as already emphasized).


\section{Some preliminary facts}\label{sec3}

We start with the existence and uniqueness of traveling fronts $(c,\varphi)$ solving~\eqref{eqvarphi} and arising in alternative~(iii) of Theorem~\ref{th1}.

\begin{lemma}\label{lem1}
Let $0<m<M$ and $f$ be a $C^1([0,M])$ function such that $f(0)=f(M)=0$, $f'(0)<0$, $f>0$ in $(m,M)$, $F<0$ in $(0,m]$ and $F(M)>0$. Then there are a unique~$c\in\R$ and $\varphi:\R\to(0,M)$ of class $C^2(\R)$ solving
\be\label{eqvarphi2bis}
\varphi''+c\varphi'+f(\varphi)=0\hbox{ in }\R,\ \varphi(-\infty)=M,\hbox{ and }\varphi(+\infty)=0,
\ee
where the uniqueness of $\varphi$ is understood up to shifts. Furthermore, $c>0$ and $\varphi'<0$ in $\R$.
\end{lemma}

The result is expected since it has been well known under some additional assumptions on~$f$. However we are not aware of a suitable reference for its proof, which is therefore sketched here for the sake of completeness.

\begin{proof}
First of all, the uniqueness of a pair $(c,\varphi)$ is a direct consequence of \cite[Corollary~2.3]{FM} and the property $\varphi'<0$ in $\R$ follows from~\cite[Lemma~2.1]{FM}. Furthermore, $c>0$ by integra\-ting~\eqref{eqvarphi2bis} against $\varphi'$ over $\R$ and using the assumption $F(M)>0$.

Let us now show the existence of a pair $(c,\varphi)$ solving~\eqref{eqvarphi2bis}. From the assumptions made on $f$, it is easy to check that there are a decreasing sequence $(\epsilon_n)_{n\in\N}$ in $(0,m)$ converging to~$0$, and a sequence $(\overline{f}_n)_{n\in\N}$ such that each function
$$\overline{f}_n:[\epsilon_n,M+\epsilon_n]\to\R$$
is of class $C^1([\epsilon_n,M+\epsilon_n])$, the sequence $(\|\overline{f}_n\|_{C^1([\epsilon_n,M+\epsilon_n])})_{n\in\N}$ is bounded, and for each $n\in\N$, there holds: $\overline{f}_n(\epsilon_n)=\overline{f}_n(M+\epsilon_n)=0$, $\overline{f}'_n(\epsilon_n)<0$, $\overline{f}'_n(M\!+\!\epsilon_n)<0$, the zeroes of~$\overline{f}_n$ are all non-degenerate (that is, $\big\{s\in[\epsilon_n,M+\epsilon_n]:\overline{f}_n(s)=\overline{f}'_n(s)=0\big\}=\emptyset$), together with~$\overline{f}_n\ge\overline{f}_{n'}$ in $[\epsilon_n,M+\epsilon_{n'}]$ if $n\le n'$, $\overline{f}_n\ge f$ in $[\epsilon_n,M]$, and $\max_{[\epsilon_n,M]}|\overline{f}_n-f|\to0$ as~$n\to+\infty$. Furthermore, even if it means considering a subsequence, one can always assume that, for each $n\in\N$,
\be\label{Fn}
\overline{f}_n>0\hbox{ in }(m,M+\epsilon_n),\ \int_{\epsilon_n}^s\!\overline{f}_n(\sigma)\,d\sigma<0\hbox{ for all }s\in(\epsilon_n,m],\hbox{ and}\int_{\epsilon_n}^{M+\epsilon_n}\!\overline{f}_n(\sigma)\,d\sigma>0.
\ee

For each $n\in\N$, it then follows from~\cite[Theorem~2.8]{FM} that there are $p\in\N$, some real numbers $\epsilon_n=a_0<a_1<\cdots<a_p=M+\epsilon_n$ and
\be\label{c1p}
c_1\ge\cdots\ge c_p
\ee
such that, for each $j\in\{1,\cdots,p\}$, there is a $C^2(\R)$ function $\varphi_j:\R\to(a_{j-1},a_j)$ satisfying
\be\label{varphij}
\varphi_j''+c_j\varphi_j'+\overline{f}_n(\varphi_j)=0\hbox{ in }\R,\ \varphi_j'<0\hbox{ in }\R,\ \varphi_j(-\infty)=a_j,\hbox{ and }\varphi_j(+\infty)=a_{j-1}.
\ee
Notice that the quantities $p$, $a_j$, $c_j$ and $\varphi_j$ actually depend on $n$, that $\overline{f}_n(a_j)=0$ for all $0\le j\le p$, and that the family $(c_j,\varphi_j)_{1\le j\le p}$ is then a stacked combination of traveling fronts with non-increasing speeds for the reaction term $\overline{f}_n$. Since $a_{p-1}\le m$ and $\int_{a_{p-1}}^{M+\epsilon_n}\overline{f}_n(\sigma)d\sigma>0$ by~\eqref{Fn}, integrating~\eqref{varphij} with $j=p$ against $\varphi'_p$ over $\R$ implies that $c_p>0$. Furthermore, if~$p\ge2$, one would have $a_1\le m$ and then $c_1<0$ by using again~\eqref{Fn} and~\eqref{varphij} with $j=1$, contradicting~\eqref{c1p}. Therefore, $p=1$ and there is then a solution $(\overline{c}_n,\overline{\varphi}_n)$ of
$$\overline{\varphi}_n''+\overline{c}_n\overline{\varphi}_n'+\overline{f}_n(\overline{\varphi}_n)=0\hbox{ in }\R,\ \overline{\varphi}_n'<0\hbox{ in }\R,\ \overline{\varphi}_n(-\infty)=M+\epsilon_n,\hbox{ and }\overline{\varphi}_n(+\infty)=\epsilon_n,$$
with $\overline{c}_n>0$.

Now, if $n<n'$, then $M+\epsilon_n>M+\epsilon_{n'}>\epsilon_n>\epsilon_{n'}$ and, up to shifts, one has~$\overline{\varphi}_n\ge\overline{\varphi}_{n'}$ in~$\R$ with equality at a point $\xi$ such that $\overline{\varphi}_n(\xi)=\overline{\varphi}_{n'}(\xi)\in(\epsilon_n,M+\epsilon_{n'})$. On the other hand, if $\overline{c}_n\le\overline{c}_{n'}$, then one would have
$$\overline{\varphi}_n''+\overline{c}_{n'}\overline{\varphi}_n'+\overline{f}_{n'}(\overline{\varphi}_n)\le\overline{\varphi}_n''+\overline{c}_n\overline{\varphi}_n'+\overline{f}_n(\overline{\varphi}_n)=0=\overline{\varphi}_{n'}''+\overline{c}_{n'}\overline{\varphi}_{n'}'+\overline{f}_{n'}(\overline{\varphi}_{n'})$$
in the open interval $I=\{x\in\R:\epsilon_n<\overline{\varphi}_n(x)<M+\epsilon_{n'}\}$. In other words, the functions $\overline{\varphi}_n$ and $\overline{\varphi}_{n'}$ are respectively a super-solution and a solution of the same elliptic equation in $I$, with $\overline{\varphi}_n\ge\overline{\varphi}_{n'}$ in $\R\supset I$. Since $\xi\in I$, it then follows from the strong maximum principle that $\overline{\varphi}_n(x)=\overline{\varphi}_{n'}(x)$ for all $x\in I$. Since $I$ is of the type $I=(\zeta,+\infty)$ for some $\zeta\in\R$, one gets a contradiction by passing to the limit in $\overline{\varphi}_n(x)=\overline{\varphi}_{n'}(x)$ as $x\to+\infty$. As a consequence, $\overline{c}_n>\overline{c}_{n'}$ and the sequence $(\overline{c}_n)_{n\in\N}$ is decreasing, with $\overline{c}_n>0$ for every $n\in\N$.

Finally, there is $c\in\R$ such that $\overline{c}_n\to c\ge0$ as $n\to+\infty$. Let $\eta$ be any real number in~$(0,m)$ such that $f<0$ in $(0,\eta]$. Up to shifts, one can assume without loss of generality that $\overline{\varphi}_n(0)=\eta$ for all $n$ large enough (such that $\epsilon_n<\eta$). From standard elliptic estimates, the functions $\overline{\varphi}_n$ converge in $C^2_{loc}(\R)$ to a $C^2(\R)$ function $\varphi$ such that $\varphi(0)=\eta$, $0\le\varphi\le M$ in~$\R$, $\varphi'\le0$ in $\R$, and
\be\label{eqvarphi5}
\varphi''+c\varphi'+f(\varphi)=0\hbox{ in }\R.
\ee
From standard elliptic estimates, $\varphi'(\pm\infty)=\varphi''(\pm\infty)=f(\varphi(\pm\infty))=0$ and the choice of $\eta$ yields $\varphi(+\infty)=0$, hence $\varphi'<0$ in $\R$ from the strong maximum principle and $0<\varphi(-\infty)\le M$. Integrating~\eqref{eqvarphi5} against $\varphi'$ over $\R$ implies that $F(\varphi(-\infty))=c\int_{\R}(\varphi')^2\ge0$. Therefore, the assumptions on $f$ imply that $\varphi(-\infty)=M$. In other words, the pair $(c,\varphi)$ solves~\eqref{eqvarphi2bis}, and the proof of Lemma~\ref{lem1} is thereby complete.
\end{proof}

\begin{remark}\label{rem32}{\rm
The arguments used in the proof of Lemma~\ref{lem1} also lead to the approximation of the unique speed $c$ from below. Namely, as above, there are a decreasing sequence $(\epsilon_n)_{n\in\N}$ in $(0,m)$ converging to $0$, and a sequence $(\underline{f}_n)_{n\in\N}$ such that each function
$$\underline{f}_n:[-\epsilon_n,M-\epsilon_n]\to\R$$
is of class $C^1([-\epsilon_n,M-\epsilon_n])$, the sequence $(\|\underline{f}_n\|_{C^1([-\epsilon_n,M-\epsilon_n])})_{n\in\N}$ is bounded, and for each~$n\in\N$, there holds: $\underline{f}_n(-\epsilon_n)=\underline{f}_n(M-\epsilon_n)=0$, $\underline{f}'_n(-\epsilon_n)<0$, $\underline{f}'_n(M-\epsilon_n)<0$, the zeroes of $\underline{f}_n$ are all non-degenerate, together with $\underline{f}_n\le\underline{f}_{n'}$ in $[-\epsilon_{n'},M-\epsilon_n]$ if $n\le n'$, $\underline{f}_n\le f$ in~$[0,M-\epsilon_n]$, and $\max_{[0,M-\epsilon_n]}|\underline{f}_n-f|\to0$ as $n\to+\infty$. Lastly, after fixing a real number $m'\in(m,M)$ such that $F<0$ in $(0,m']$, one can assume without loss of generality that, for each $n\in\N$,
$$\underline{f}_n>0\hbox{ in }[m',M-\epsilon_n),\ \int_{-\epsilon_n}^s\!\underline{f}_n(\sigma)\,d\sigma<0\hbox{ for all }s\in(-\epsilon_n,m'],\hbox{ and }\int_{-\epsilon_n}^{M-\epsilon_n}\!\underline{f}_n(\sigma)\,d\sigma>0.$$
As in the proof of Lemma~\ref{lem1}, one can then show that, for each $n\in\N$, there is a solution $(\underline{c}_n,\underline{\varphi}_n)$ of $\underline{\varphi}_n''+\underline{c}_n\underline{\varphi}_n'+\underline{f}_n(\underline{\varphi}_n)=0$ in $\R$, $\underline{\varphi}_n'<0$ in $\R$, $\underline{\varphi}_n(-\infty)=M-\epsilon_n$, $\underline{\varphi}_n(+\infty)=-\epsilon_n$, and $\underline{c}_n>0$. Furthermore, the sequence $(\underline{c}_n)_{n\in\N}$ is increasing, and $\underline{c}_n<c$ for all $n\in\N$, with the same arguments as in the proof of Lemma~\ref{lem1}. Finally, there is $\underline{c}\le c$ such that~$\underline{c}_n\to\underline{c}$ as~$n\to+\infty$, and there is a $C^2(\R)$ solution $\underline{\varphi}$ of~\eqref{eqvarphi2bis} with speed $\underline{c}$ instead of $c$. The uniqueness of $(c,\varphi)$ then yields $\underline{c}=c$, hence $\underline{c}_n\to c$ as~$n\to+\infty$.}
\end{remark}

The second preliminary result is concerned with the existence of solutions of some semilinear elliptic equations in large balls.

\begin{lemma}\label{lem2}
Let $\alpha<\beta$ be two real numbers and $g:[\alpha,\beta]\to\R$ be a $C^1([\alpha,\beta])$ function such that $g(\alpha)=g(\beta)=0$ and
\be\label{hypg}
G(\beta):=\int_\alpha^\beta g(\sigma)d\sigma>\int_\alpha^sg(\sigma)d\sigma=:G(s)\hbox{ for all }s\in[\alpha,\beta).
\ee
Then, for each $\nu\in(\alpha,\beta)$, there are $R>0$ and a $C^2(\overline{B_R})$ function $\psi$ such that
\be\label{eqpsi}\left\{\baa{rcll}
\Delta\psi+g(\psi) & \!=\! & 0 & \hbox{in }\overline{B_{R}},\vspace{3pt}\\
\alpha\ \le\ \psi & \!<\! & \beta & \hbox{in }\overline{B_{R}},\vspace{3pt}\\
\psi & \!=\! & \alpha & \hbox{on }\partial B_{R},\vspace{3pt}\\
\displaystyle\mathop{\max}_{\overline{B_{R}}}\psi\ =\ \psi(0) & \!>\! & \nu.& \eaa\right.
\ee
\end{lemma}

\begin{proof}
The proof is standard, based on~\cite{BL1}, so we just sketch it. Let $\overline{g}:\R\to\R$ be the function defined by $\overline{g}(s)=g(s)$ for $s\in[\alpha,\beta]$, and $\overline{g}(s)=0$ for $s\in\R\!\setminus\![\alpha,\beta]$, and let us still call $G$ the primitive of $\overline{g}$ vanishing at $\alpha$. For each $r>0$, there is a minimizer~$\psi_r\in \alpha+H^1_0(B_{r})$ of the Lagrangian $J_r$ defined in $\alpha+H^1_0(B_{r})$ by
$$J_r(\varphi)=\frac{1}{2}\int_{B_{r}}|\nabla\varphi(x)|^2\,dx-\int_{B_{r}}G(\varphi(x))\,dx,\ \ J_r(\psi_r)=\min_{\varphi\in \alpha+H^1_0(B_{r})}J_r(\varphi),$$
Owing to the definitions of $\overline{g}$ and $G$, one can assume without loss of generality that $\psi_r$ ranges in $[\alpha,\beta]$, hence $\psi_r$ is of class $C^2(\overline{B_r})$ from elliptic estimates and  it solves~$\Delta\psi_r+g(\psi_r)=0$ in $\overline{B_r}$ with $\psi_r=\alpha$ on $\partial B_r$. The strong elliptic maximum principle also yields $\psi_r<\beta$ in~$\overline{B_r}$ and, either $\psi_r\equiv\alpha$ in $\overline{B_r}$, or $\psi_r>\alpha$ in $B_{r}$. In both cases,~$\psi_r$ is a radially symmetric and nonincreasing function of $|x|$ (from~\cite{GNN1} in the latter). In particular, $\max_{\overline{B_{r}}}\psi_r\!=\!\psi_r(0)\!\in\![\alpha,\beta)$.

Let us now show that
\be\label{limitbeta}
\max_{\overline{B_{r}}}\psi_r=\psi_r(0)\to\beta\ \hbox{ as }r\to+\infty,
\ee
which will then provide $R>0$ and a solution $\psi$ of~\eqref{eqpsi}, given a fixed real number $\nu\in(\alpha,\beta)$. Assume by way of contradiction that there are $\theta\in(\alpha,\beta)$, a sequence $(r_k)_{k\in\N}\to+\infty$ and a sequence~$(\psi_{r_k})_{k\in\N}$ of $C^2(\overline{B_{r_k}})$ functions such that each $\psi_{r_k}:\overline{B_{r_k}}\to[\alpha,\beta)$ minimizes $J_{r_k}$ in $\alpha+H^1_0(B_{r_k})$ and $\max_{\overline{B_{r_k}}}\psi_{r_k}=\psi_{r_k}(0)\le\theta<\beta$. From the assumptions made on $g$, there is $\delta>0$ such that $G(s)\le G(\beta)-\delta$ for all $s\in[\alpha,\theta]$. Hence, $J_{r_k}(\psi_{r_k})\ge(\delta-G(\beta))\,\alpha_Nr_k^N$ for all $k\in\N$, where $\alpha_N>0$ denotes the Lebesgue measure of the $N$-dimensional unit ball~$B_1$. On the other hand, after assuming without loss of generality that $r_k>1$ for every $k\in\N$, consider the function~$\varphi_k\in\alpha+H^1_0(B_{r_k})$ defined by $\varphi_k(x)=\beta$ for $x\in B_{r_k-1}$ and~$\varphi_k(x)=\alpha+(\beta-\alpha)(r_k-|x|)$ for $x\in\overline{B_{r_k}}\!\setminus\!B_{r_k-1}$. For each~$k\in\N$, one has
$$\baa{rcl}
J_{r_k}(\psi_{r_k}) & \le & \displaystyle J_{r_k}(\varphi_k)=\frac{\alpha_N}{2}\,(r_k^N-(r_k-1)^N)-G(\beta)\,\alpha_N\,(r_k-1)^N-\int_{B_{r_k}\setminus B_{r_k-1}}G(\varphi_k(x))\,dx\vspace{3pt}\\
& \le & \displaystyle\alpha_N\,\Big(\frac{1}{2}+\max_{[\alpha,\beta]}|G|\Big)\,(r_k^N-(r_k-1)^N)-G(\beta)\alpha_N(r_k-1)^N.\eaa$$
This implies that
\[
\delta\,r_k^N\le \left(\dfrac 12+\max_{[\alpha,\beta]}|G|+G(\beta)\right)(r_k^N-(r_k-1)^N).
\]
It contradicts $\lim_{k\to+\infty}r_k=+\infty$, since $\delta>0$. As a consequence,~\eqref{limitbeta} holds and the proof of Lemma~\ref{lem2} is thereby complete.
\end{proof}


\section{Proof of Theorem~\ref{th1}}\label{sec4}

This section is devoted to the proof of Theorem~\ref{th1}. Throughout it, one assumes that $f$ satisfies~\eqref{hypf} and~\eqref{hypf2}, and $u$ denotes a positive bounded solution of~\eqref{eq:ACN} satisfying~\eqref{localized}. Sections~\ref{sec41} and~\ref{sec42} are concerned with the behaviors of $u$ as $t\to-\infty$ and~$t\to+\infty$, respectively.


\subsection{The behavior of $u$ as $t\to-\infty$}\label{sec41}

We here establish~\eqref{conv1} and further integral pro\-perties of the solution $u$. First of all, it follows from~\cite[Theorem~1.1]{P1} that there is a point~$x_0\in\R^N$ such that
\be\label{defx0}\left\{\baa{ll}
u(t,x)=u(t,y) & \hbox{for all }(t,x,y)\in\R\times\R^N\times\R^N\hbox{ with }|x-x_0|=|y-x_0|,\vspace{3pt}\\
\nabla u(t,x)\cdot(x-x_0)<0 & \hbox{for all }(t,x)\in\R\times\R^N\hbox{ with }x\neq x_0.\eaa\right.
\ee
Furthermore,~\cite[Corollary~2.5]{P1} implies that there are some positive constants $C$ and $\nu$ such that
\be\label{inequ-nux}
0<u(t,x)\le C\,e^{-\nu|x|}\ \hbox{ for all }(t,x)\in(-\infty,0]\times\R^N.
\ee
Denote
$$M_0=\|u\|_{L^\infty(\R\times\R^N)}>0$$
and $L=\nu^2+\max_{[0,M_0]}|f'|$. For any unit vector $e$ of $\R^N$, the function $\overline{u}(t,x)=C\,e^{-\nu x\cdot e+Lt}$ satisfies $\overline{u}_t(t,x)-\Delta\overline{u}(t,x)-f(\overline{u}(t,x))\ge0$ for any $(t,x)\in\R\times\R^N$ such that $\overline{u}(t,x)\le M_0$. Therefore, the maximum principle implies that $u(t,x)\le\overline{u}(t,x)$ for all $(t,x)\in[0,+\infty)\times\R^N$ and for any unit vector $e$, hence $u(t,x)\le C\,e^{-\nu|x|+Lt}$ for all $(t,x)\in[0,+\infty)\times\R^N$. By combining this inequality with \eqref{inequ-nux}, for every $T\in\R$, there is a real number $C_T>0$ such that
\be\label{expdecay}
0<u(t,x)\le C_T\,e^{-\nu|x|}\ \hbox{ for all }(t,x)\in(-\infty,T]\times\R^N.
\ee
From standard parabolic estimates, one also infers that, for every $T\in\R$, there is a real number~$C'_T>0$ such that
\be\label{exp}
u(t,x)\!+\!|u_t(t,x)|\!+\!|\nabla u(t,x)|\!+\!\!\sum_{1\le i,j\le N}\!\!\!\!|u_{x_ix_j}(t,x)|\le C'_Te^{-\nu|x|}\hbox{ for all }(t,x)\!\in\!(-\infty,T]\!\times\!\R^N.
\ee

In particular, the action
\be\label{defE}
E[u(t,\cdot)]=\int_{\R^N}\Big(\frac{|\nabla u(t,x)|^2}{2}-F(u(t,x)\Big)\,dx
\ee
is well defined at each time $t\in\R$, and Lebesgue's dominated convergence theorem implies that the function $t\mapsto E[u(t,\cdot)]$ is of class $C^1(\R)$ with
\be\label{lyapounov}
\frac{d}{dt}E[u(t,\cdot)]=-\int_{\R^N}(u_t(t,x))^2\,dx\le0
\ee
for every $t\in\R$. Furthermore,~\eqref{exp} yields $\sup_{t\le T}|E[u(t,\cdot)]|<+\infty$ for every $T\in\R$, and there is then $\ell\in\R$ such that
\be\label{Eutl}
E[u(t,\cdot)]\to\ell\ \hbox{ as }t\to-\infty.
\ee

Consider now any sequence $(t_n)_{n\in\N}$ converging to $-\infty$, and denote
$$u_n(t,x)=u(t+t_n,x)$$
for $(t,x)\in\R\times\R^N$. From~\eqref{exp} and further standard parabolic estimates, there is a classical nonnegative bounded solution $u_\infty$ of~\eqref{eq:ACN} such that, up to extraction of a subsequence, $u_n\to u_\infty$ in $C^{1,2}_{t,x}$ locally in $\R\times\R^N$, together with
$$\|u_n(t,\cdot)-u_\infty(t,\cdot)\|_{L^\infty(\R^N)}\to0\hbox{ and }E[u_n(t,\cdot)]\to E[u_\infty(t,\cdot)]\ \hbox{ as }n\to+\infty,$$
for every $t\in\R$. Notice also that $u_\infty$ satisfies~\eqref{exp} with the constant $C'_0$ in the whole set~$\R\!\times\!\R^N$. Since $E[u_n(t,\cdot)]=E[u(t+t_n,\cdot)]\to\ell$ as $n\to+\infty$, for every $t\in\R$, one infers that $E[u_\infty(t,\cdot)]=\ell$ for every $t\in\R$, hence $(u_\infty)_t\equiv 0$ in $\R\times\R^N$ from~\eqref{lyapounov} applied to $u_\infty$. As a consequence, $u_\infty$ is a bounded nonnegative steady state solving
$$\left\{\baa{l}
\Delta u_\infty+f(u_\infty)=0\hbox{ and }u_\infty\ge0\hbox{ in }\R^N,\vspace{3pt}\\
u_\infty(x)\to0\hbox{ as }|x|\to+\infty.\eaa\right.$$
From the elliptic maximum principle, it follows that either $u_\infty\equiv 0$ in $\R^N$, or $u_\infty>0$ in~$\R^N$. In the former case, one has $u(t_n,\cdot)=u_n(0,\cdot)\to0$ as $n\to+\infty$ uniformly in $\R^N$, hence~$0<u(t_n,\cdot)\le\eta$ in $\R^N$ for all $n$ large enough, where $\eta$ is a positive real number such that
\be\label{defeta}
f<0\hbox{ in }(0,\eta].
\ee
Thus, for all $t\in\R$, one gets that $0<u(t,\cdot)=u(t-t_n+t_n,\cdot)\le\zeta(t-t_n)$ in $\R^N$ for all~$n$ large enough, where~$\zeta$ obeys $\zeta(0)=\eta$ and $\zeta'(t)=f(\zeta(t))$ for all $t\ge0$. Since $\zeta(+\infty)=0$ and $\lim_{n\to+\infty}t_n=-\infty$, it follows that $u(t,\cdot)\le0$ in $\R^N$ for all $t\in\R$, a contradiction. Therefore,~$u_\infty>0$ is a steady state solving~\eqref{steady}, namely $u_\infty\in\mathcal{E}$. In particular, $\mathcal E$ is not empty.

Since in the previous paragraph the sequence $(t_n)_{n\in\N}$ converging to $-\infty$ was arbitrary, one  concludes that
$$\inf_{\phi\in\mathcal{E}}\|u(t,\cdot)-\phi\|_{L^\infty(\R^N)}\to0\ \hbox{ as }t\to-\infty,$$
namely~\eqref{conv1} has been proven. The observations of the previous paragraph also imply that
\be\label{Ephil}
E[\phi]=\ell
\ee
for every $\phi\in\mathcal{E}$ belonging to the $\alpha$-limit set of $u$.

\begin{remark}
Remember that, from assumption~\eqref{hypf2}, the map $(\emptyset\neq)\,\mathcal E\ni\phi\mapsto m_\phi$ takes a constant value, that is, there is $m>0$ such that
\be\label{mmphi}
m=m_\phi>0
\ee
for all $\phi\in\mathcal E$. The quantity $M\in(m,+\infty]$ defined in~\eqref{defMphi} and~\eqref{MMphi}-\eqref{defM} is such that
$$f>0\hbox{ in }(m,M)$$
from~\eqref{mMphi}. In the present remark, we claim that
\be\label{0uM}
0<u(t,x)<M\ \hbox{ for all }(t,x)\in\R\times\R^N.
\ee
By assumption, $u$ is positive. So there is nothing to show if $M=+\infty$. Assume now that $M<+\infty$. From the above proof of~\eqref{conv1}, there is $\phi\in\mathcal E$ and a sequence $(t_n)_{n\in\N}$ converging to $-\infty$ such that $\|u(t_n,\cdot)-\phi\|_{L^\infty(\R^N)}\to0$ as $n\to+\infty$. Since $\max_{\R^N}\!\phi<M$ by~\eqref{defMphi} and~\eqref{MMphi}, one has $u(t_n,\cdot)<M$ in $\R^N$ for all $n$ large enough. Since $f(M)=0$, it then follows from the maximum principle (applied for $n$ large enough) that, for each $t\in\R$, $u(t,\cdot)=u(t-t_n+t_n,\cdot)<M$ in $\R^N$, that is,~\eqref{0uM} holds. 
\end{remark}


\subsection{The behavior of $u$ as $t\to+\infty$}\label{sec42}

In the section, we consider the behavior of the entire solution $u$ as $t\to+\infty$. The proof is divided into five main steps.\hfill\break

{\noindent{\it{Step 1: two key-lemmas.}}} The proof of the dichotomy as $t\to+\infty$ between the uniformly localized solutions and the spreading solutions is based on two key-lemmas. The first one gives a sufficient condition for the finiteness and attractiveness of the quantity $M\in(m,+\infty]$ defined in~\eqref{defMphi} and~\eqref{MMphi}-\eqref{defM}.

\begin{lemma}\label{lem3}
For every $\epsilon>0$, there is a real number $\rho_\epsilon>0$ such that, if
\be\label{large}
u(t_0,\cdot)\ge m+\epsilon\hbox{ in }\overline{B(y_0,\rho_\epsilon)}
\ee
for some $(t_0,y_0)\in\R\times\R^N$, then
$$M<+\infty$$
and
\be\label{spreading2}
\max_{|x|\le\gamma t}|u(t,x)-M|\to0\hbox{ as }t\to+\infty
\ee
for some $\gamma>0$.\footnote{We point out that, when $\ep> M_0-m$, $\rho_\ep$ can be arbitrary because~\eqref{large} is not fulfilled in that case.}
\end{lemma}

\begin{remark}
For a function $f$ such that $M$ is a priori assumed to be finite, the conclusion~\eqref{spreading2} can also be viewed as a consequence of~\cite[Lemma~2.4]{DP}, which is based on the existence of approximated planar fronts defined in bounded intervals (see also~\cite[Theorem~3.2]{FM} and~\cite[Theorem~6.2]{AW} for related results with more specific nonlinearities $f$ in the one- and multi-dimensional cases). We here both show~\eqref{spreading2} and the finiteness of $M$ under assumption~\eqref{large}. Moreover, the proof of~\eqref{spreading2} given below differs from that of~\cite[Lemma~2.4]{DP} as it is based on Lemma~\ref{lem2} and on the existence of compactly supported steady states.
\end{remark}

\begin{proof}[Proof of Lemma~$\ref{lem3}$] Let $\epsilon>0$ be fixed throughout the proof. Assume first, by way of contradiction, that $M=+\infty$, and that there exists a sequence $(t_n,y_n)_{n\in\N}$ in $\R\times\R^N$ such that $u(t_n,\cdot)\ge m+\epsilon$ in $\overline{B(y_n,n)}$. From~\eqref{mMphi}, it then follows that $f>0$ in $(m,+\infty)$. From standard parabolic estimates, the functions
$$u_n:(t,x)\mapsto u_n(t,x)=u(t+t_n,x+y_n)$$
converge in $C^{1,2}_{t,x}$ locally in $\R\times\R^N$, up to extraction of a subsequence, to a nonnegative bounded solution $u_\infty$ of~\eqref{eq:ACN} such that $u_\infty(0,\cdot)\ge m+\epsilon$ in~$\R^N$. Hence, $u_\infty(t,\cdot)\ge\varsigma(t)$ in~$\R^N$ for all $t\ge0$, where $\varsigma$ obeys
\be\label{defvarsigma}\left\{\baa{l}
\varsigma'(t)=f(\varsigma(t))\ \hbox{ for }t\ge0,\vspace{3pt}\\
\varsigma(0)=m+\epsilon.\eaa\right.
\ee
This is impossible since $u_\infty$ is bounded while $f>0$ in $[m+\epsilon,+\infty)$. As a consequence, there is $\varrho_\epsilon>0$ such that if
$$u(t_0,\cdot)\ge m+\epsilon\hbox{ in $\overline{B(y_0,\varrho_\epsilon)}$}$$
for some $(t_0,y_0)\in\R\times\R^N$, then
$$M<+\infty.$$

We now claim that there is $\rho_\epsilon\in[\varrho_\epsilon,+\infty)$ such that, if condition~\eqref{large} is fulfilled for some~$(t_0,y_0)\in\R\times\R^N$, then $M<+\infty$ (from the previous paragraph) and
\be\label{spreading4}
u(t,\cdot)\to M\hbox{ as }t\to+\infty\hbox{ locally uniformly in $\R^N$.}
\ee
Assume not. Then there is a sequence $(\tau_n,z_n)_{n\in\N,\,n\ge\varrho_\epsilon}$ in $\R\times\R^N$ such that $u(\tau_n,\cdot)\ge m+\epsilon$ in $\overline{B(z_n,n)}$ (hence, $M<+\infty$) and
\be\label{contra}
u(t,\cdot)\not\to M\hbox{ as $t\to+\infty$ locally uniformly in $\R^N$}.
\ee
Notice that~\eqref{0uM} then implies that
$$m<m+\epsilon<M.$$
On the other hand, since $F<0$ in $(0,m]$ by~\eqref{hypf2} and~\eqref{mmphi}, since $F(M)>0$ by~\eqref{MMphi}-\eqref{defM} and since $f>0$ in $(m,M)$ by~\eqref{mMphi}, one infers that $F(s)<F(M)$ for all $s\in[0,M)$. Lemma~\ref{lem2} applied with $\alpha=0$, $\beta=M$, $\gamma=m$ and $g=f$ yields the existence of $R>0$ and a~$C^2(\overline{B_R})$ function $\psi$ such that
\be\label{eqpsi2}\left\{\baa{rcll}
\Delta\psi+f(\psi) & \!=\! & 0 & \hbox{in }\overline{B_R},\vspace{3pt}\\
0\ \le\ \psi & \!<\! & M & \hbox{in }\overline{B_R},\vspace{3pt}\\
\psi & \!=\! & 0 & \hbox{on }\partial B_R,\vspace{3pt}\\
\displaystyle M>\mathop{\max}_{\overline{B_R}}\psi\ =\ \psi(0) & \!>\! & m>0.& \eaa\right.
\ee
Let $\varsigma$ be the solution of~\eqref{defvarsigma}. Since~$m+\epsilon<M$ and $f>0$ in $(m,M)$ with $f(M)=0$, one has $\varsigma(t)\to M$ as $t\to+\infty$. Hence, there is a positive real number $T>0$ such that
\be\label{defT}
\varsigma(T)>\psi(0).
\ee
Up to extraction of a subsequence, the functions
$$v_n:(t,x)\mapsto v_n(t,x)=u(t+\tau_n,x+z_n)$$
converge in $C^{1,2}_{t,x}$ locally in $\R\times\R^N$ to a nonnegative bounded solution $v_\infty$ of~\eqref{eq:ACN} such that $v_\infty(0,\cdot)\ge m+\epsilon$ in $\R^N$. Hence, $v_\infty(T,\cdot)\ge\varsigma(T)>\psi(0)$. It then follows from the last line in~\eqref{eqpsi2} that there is $n_0\in\N$ (with $n_0\ge\varrho_\epsilon$) such that
\be\label{uTw}
u(T+\tau_{n_0},\cdot+z_{n_0})>\psi\ \hbox{ in }\overline{B_R}.
\ee

Let then $w$ be the solution of the equation $w_t=\Delta w+f(w)$ in $(0,+\infty)\times\R^N$ with initial condition given by
\be\label{defw0}
w(0,x)=\left\{\baa{ll}\psi(x) & \hbox{if }x\in\overline{B_R},\vspace{3pt}\\
0 & \hbox{if }x\in\R^N\setminus\overline{B_R}.\eaa\right.
\ee
Since $\psi$ satisfies~\eqref{eqpsi2} and $f(M)=f(0)=0$, the parabolic maximum principle implies that~$0<w<M$ in $(0,+\infty)\times\R^N$ and $w$ is increasing with respect to $t$ in $[0,+\infty)\times\R^N$. From standard parabolic estimates and uniqueness of the limit, there is a $C^2(\R^N)$ solution~$w_\infty$ of~$\Delta w_\infty+f(w_\infty)=0$ in $\R^N$ with $0<w_\infty\le M$ in $\R^N$ and $w_\infty>\psi$ in $\overline{B_R}$. Let then $e$ be any unit vector of $\R^N$. By continuity, there is $s_0>0$ such that $w_\infty>\psi(\cdot-se)$ in $\overline{B(se,R)}$ for all $s\in[0,s_0]$. Calling
$$s^*=\sup\big\{s>0:w_\infty>\psi(\cdot-s'e)\hbox{ in }\overline{B(s'e,R)}\hbox{ for all }s'\in[0,s]\big\}\in[s_0,+\infty]$$
and assuming that $s^*<+\infty$, one would have $w_\infty\ge\psi(\cdot-s^*e)$ in $\overline{B(s^*e,R)}$ with equality somewhere at a point $x^*\in\overline{B(s^*e,R)}$. Since $w_\infty>0$ in $\R^N$ and $\psi=0$ on $\partial B_R$, the point~$x^*$ would be an interior point in $B(s^*e,R)$, hence $w_\infty\equiv\psi(\cdot-s^*e)$ in $\overline{B(s^*e,R)}$ from the strong maximum principle. This is impossible on $\partial B(s^*,R)$. Thus $s^*=+\infty$ and, since the unit vector $e$ was arbitrary, one gets that $(M\ge)\,w_\infty>\psi(0)>m$ in $\R^N$. The positivity of~$f$ in $(m,M)$ then implies that $w_\infty\equiv M$ in $\R^N$, hence
\be\label{wM}
w(t,\cdot)\to M\hbox{ as }t\to+\infty\hbox{ locally uniformly in }\R^N.
\ee
Together with~\eqref{uTw}-\eqref{defw0} and the maximum principle, one infers that
$$\liminf_{t\to+\infty}\Big(\min_Ku(t,\cdot)\Big)\ge M$$
for any compact set $K\subset\R^N$, and finally $u(t,\cdot)\to M$ as $t\to+\infty$ locally uniformly in $\R^N$ from~\eqref{0uM}. This contradicts~\eqref{contra}.

As a conclusion of the previous paragraph, there is a positive real number $\rho_\epsilon(\ge\varrho_\epsilon)$ such that if condition~\eqref{large} is fulfilled for some $(t_0,y_0)\in\R\times\R^N$, then $M<+\infty$ and~\eqref{spreading4} holds. Let us finally show that this implies the stronger property~\eqref{spreading2}: $\max_{|x|\le\gamma t}|u(t,x)-M|\to0$ as $t\to+\infty$, for some~$\gamma>0$. To do so, observe on the one hand that, since $w(t,\cdot)\to M$ as~$t\to+\infty$ locally uniformly in $\R^N$ by~\eqref{wM} and since $\max_{\overline{B_R}}\psi<M$, there is a time $\tau>0$ such that
$$w(\tau,\cdot)\ge\psi(\cdot-se)\ \hbox{ in }\overline{B(se,R)}\ \hbox{ for every unit vector $e$ and every $s\in[0,1]$}.$$
In other words, $w(\tau,\cdot)\ge w(0,\cdot-se)$ in $\R^N$ for every unit vector $e$ and every $s\in[0,1]$. From the maximum principle, one gets that $w(2\tau,\cdot)\ge w(\tau,\cdot-se)\ge w(0,\cdot-2se)$ in $\R^N$ for every unit vector $e$ and every $s\in[0,1]$. Hence, by an immediate induction,
$$w(k\tau,\cdot)\ge w(0,\cdot-se)\hbox{ in }\R^N\hbox{ for every $k\in\N$, every unit vector $e$, and every $s\in[0,k]$}.$$
On the other hand,~\eqref{spreading4} and~\eqref{eqpsi2} yield the existence of a time $\tau_*>0$ such that
$$u(\tau_*,\cdot)\ge\psi\ \hbox{ in }\overline{B_R},$$
that is, $u(\tau_*,\cdot)\ge w(0,\cdot)$ in $\R^N$. Therefore,
$$u(\tau_*+k\tau,\cdot)\ge w(0,\cdot-se)\hbox{ in }\R^N$$
for every $k\in\N$, every unit vector $e$, and every $s\in[0,k]$. In particular,
\be\label{ineqk}
\min_{\overline{B_k}}u(\tau_*+k\tau,\cdot)\ge w(0,0)=\psi(0)\ \hbox{ for all }k\in\N.
\ee

We finally claim that
\be\label{claimspreading}
\max_{|x|\le t/(2\tau)}|u(t,x)-M|\to0\ \hbox{ as }t\to+\infty,
\ee
which will give the desired conclusion~\eqref{spreading2} with $\gamma=1/(2\tau)$. Assume by way of contradiction that~\eqref{claimspreading} does not hold. Since $0<u<M$ in $\R\times\R^N$ by~\eqref{0uM}, there are then a real number~$\theta\in[0,M)$ and a sequence $(s_n,\xi_n)_{n\in\N}$ in $(0,+\infty)\times\R^N$ such that
\be\label{deftheta}
\lim_{n\to+\infty}s_n=+\infty,\ \ \lim_{n\to+\infty}u(s_n,\xi_n)=\theta,\ \hbox{ and }\ |\xi_n|\le\frac{s_n}{2\tau}\hbox{ for all }n\in\N.
\ee
Consider any integer $j\in\N$. For all $n$ large enough, write
\be\label{snkn}
s_n=\tau_*+k_n\tau+s'_n,\ \hbox{ with $k_n\in\N$ and $s'_n\in[j\tau,(j+1)\tau)$}
\ee
(the quantities $k_n$ and $s'_n$ depend on $j$ as well, but this does not matter). Up to extraction of a subsequence, there is $s'_\infty\in[j\tau,(j+1)\tau]$ such that $s'_n\to s'_\infty$ as $n\to+\infty$. Up to extraction of another subsequence, the functions
$$U_n:(t,x)\mapsto U_n(t,x)=u(t+\tau_*+k_n\tau,x+\xi_n)$$
converge in $C^{1,2}_{t,x}$ locally in $\R\times\R^N$ to a solution $U_\infty$ of~\eqref{eq:ACN} such that $0\le U_\infty\le M$ in $\R\times\R^N$. For each $x\in\R^N$, one has $|x+\xi_n|\le k_n$ for all $n$ large enough, since $|\xi_n|\le s_n/(2\tau)$ for all $n$ and $s_n\sim k_n\tau$ as $n\to+\infty$ from~\eqref{snkn} and $\lim_{n\to+\infty}s_n=+\infty$. It then follows from~\eqref{ineqk} that $U_\infty(0,\cdot)\ge\psi(0)$ in $\R^N$, hence
$$U_\infty(t,\cdot)\ge\omega(t)\ \hbox{ in $\R^N$ for all $t\ge0$},$$
where $\omega$ obeys $\omega'(t)=f(\omega(t))$ and $\omega(0)=\psi(0)$. Since $m<\psi(0)<M$ and $f>0$ in $(m,M)$ with $f(M)=0$, one has $\omega(t)\to M$ as $t\to+\infty$ (notice that $U_\infty$ and $s'_\infty\in[j\tau,(j+1)\tau]$ depend on $j\in\N$, but $\psi(0)$ and $\omega$ do not). It also follows from~\eqref{deftheta}-\eqref{snkn} that $U_\infty(s'_\infty,0)=\theta<M$, hence $M>\theta\ge\omega(s'_\infty)$. Since $s'_\infty\in[j\tau,(j+1)\tau]$ and $\omega(+\infty)=M$, the passage to the limit as~$j\to+\infty$ in the inequality $M>\theta\ge\omega(s'_\infty)$ leads to a contradiction. As a conclusion,~\eqref{claimspreading} has been shown and the proof of Lemma~\ref{lem3} is thereby complete.
\end{proof}

The second key-lemma gives a quantitative estimate of the time the solution takes to go from $m+\epsilon$ to any value $\lambda$ less than $M$ in large balls.

\begin{lemma}\label{lem4}
Under the notations of Lemma~$\ref{lem3}$, for every $\epsilon>0$, $\lambda<M$ and $r\ge0$, there are some real numbers $\rho_{\epsilon,\lambda,r}\ge\rho_\epsilon>0$ and $T_{\epsilon,\lambda,r}>0$ such that, if
$$u(t_0,\cdot)\ge m+\epsilon\hbox{ in }\overline{B(y_0,R)}$$
for some $(t_0,y_0)\in\R\times\R^N$ and $R\ge\rho_{\epsilon,\lambda,r}$, then
$$u(t,\cdot)\ge\lambda\hbox{ in }\overline{B(y_0,R+r)}\ \hbox{ for all }t\ge t_0+T_{\epsilon,\lambda,r}.$$
\end{lemma}

\begin{proof}
Let us fix $\epsilon>0$, $\lambda<M$ and $r\ge0$, and let $\rho_\epsilon>0$ be given by Lemma~\ref{lem3}. Assume by way of contradiction that the conclusion of Lemma~\ref{lem4} does not hold. Then there exist two sequences $(R_n)_{n\in\N}$ and $(T_n)_{n\in\N}$ of positive real numbers converging to $+\infty$, and a sequence~$(t_n,y_n,z_n)_{n\in\N}$ in $\R\times\R^N\times\R^N$ such that
\be\label{defynzn}
u(t_n,\cdot)\ge m+\epsilon\hbox{ in }\overline{B(y_n,R_n)},\ z_n\in\overline{B(y_n,R_n+r)},\hbox{ and }u(t_n+T_n,z_n)<\lambda,
\ee
for all $n\in\N$. Notice that Lemma~\ref{lem3} then implies that $M<+\infty$, and that $m+\epsilon<M$ by~\eqref{0uM} and~\eqref{defynzn}.

Let now $R>0$ and $\psi\in C^2(\overline{B_R})$ be as in~\eqref{eqpsi2}, let $w$ be the solution of the Cauchy problem $w_t=\Delta w+f(w)$ in~$(0,+\infty)\times\R^N$ with initial condition $w(0,\cdot)$ given by~\eqref{defw0}, and let $\varsigma\in C^1([0,+\infty))$ and $T>0$ be defined as in~\eqref{defvarsigma} and \eqref{defT}. For any $\varrho>0$, call $v_\varrho$ the solution of $(v_\varrho)_t=\Delta v_\varrho+f(v_\varrho)$ in~$(0,+\infty)\times\R^N$ with initial condition $v_\varrho(0,\cdot)$ defined by: 
$$v_\varrho(0,x)=\left\{\begin{array}{ll}
m+\epsilon & \hbox{if $x\in B_\varrho$},\vspace{3pt}\\
0 & \hbox{if $x\in\R^N\setminus B_\varrho$.}\end{array}\right.$$
From standard parabolic estimates, there holds $v_\varrho(T,\cdot)\to\varsigma(T)\ (>\psi(0))$ as~$\varrho\to+\infty$ locally uniformly in $\R^N$ (e.g. see \cite[Theorem 4.1]{GNSY}). Hence, there is $\varrho_0>0$ such that~$v_{\varrho_0}(T,\cdot)>\psi(0)$ in~$\overline{B_R}$, and then
$$v_{\varrho_0}(T,\cdot)>w(0,\cdot)\hbox{ in $\R^N$}.$$
Since $\lambda<M$ by assumption and since $w(t,\cdot)\to M$ as $t\to+\infty$ locally uniformly in $\R^N$ by~\eqref{wM}, there is $T'>0$ such that
\be\label{defT'}
w(t,\cdot)\ge\lambda\hbox{ in $\overline{B_{r+\varrho_0}}$ for all $t\ge T'$}.
\ee 

Notice that the parameters and functions introduced in the previous paragraph do not depend on $n$. Coming back to~\eqref{defynzn}, one can assume without loss of generality that $R_n\ge\varrho_0$ for all $n\in\N$. Hence, by~\eqref{defynzn}, for each $n\in\N$, there is a point $y'_n$ such that
\be\label{defy'n}
|z_n-y'_n|\le r+\varrho_0\hbox{ and $B(y'_n,\varrho_0)\subset B(y_n,R_n)$},
\ee
and thus $u(t_n,\cdot)\ge v_{\varrho_0}(0,\cdot-y'_n)$ in $\R^N$. The maximum principle then yields
$$u(t_n+T,\cdot)\ge v_{\varrho_0}(T,\cdot-y'_n)>w(0,\cdot-y'_n)\hbox{ in $\R^N$}$$
and $u(t_n+t,\cdot)>w(t-T,\cdot-y'_n)$ in $\R^N$ for all $t\ge T$. For all $n$ large enough so that $T_n\ge T+T'$, it then follows that $u(t_n+T_n,z_n)>w(T_n-T,z_n-y'_n)\ge\lambda$ by~\eqref{defT'}-\eqref{defy'n}, contradicting the last property of~\eqref{defynzn}.

To sum up, the existence of the sequences $(R_n)_{n\in\N}$, $(T_n)_{n\in\N}$ and~$(t_n,y_n,z_n)_{n\in\N}$ is ruled out and the proof of Lemma~\ref{lem4} is thereby complete.
\end{proof}

\noindent{\it{Step 2: a dichotomy.}} Fix a real number $\eta>0$ such that~\eqref{defeta} holds. Since $f(m)=f(m_\phi)=0$ for all $\phi\in\mathcal E$, with $m_\phi>0$, one has
$$0<\eta<m.$$
Define $g(s)=-f(-s)$ for all $s\in[-m,0]$. From assumption~\eqref{hypf2}, the $C^1([-m,0])$ function~$g$ satis\-fies~\eqref{hypg} with $\alpha=-m$ and $\beta=0$. Lemma~\ref{lem2} applied with $\nu=-\eta\in(-m,0)$ then provides the existence of $R_1>0$ and of a function $\psi\in C^2(\overline{B_{R_1}})$ solving $\Delta\psi+g(\psi)=0$ and $-m\le\psi<0$ in~$\overline{B_{R_1}}$ with $\psi=-m$ on $\partial B_{R_1}$ and $\max_{\overline{B_{R_1}}}\psi=\psi(0)>-\eta$. In other words, the function $\varphi=-\psi\in C^2(\overline{B_{R_1}})$ solves
\be\label{eqvarphi3}\left\{\baa{rcll}
\Delta\varphi+f(\varphi) & \!=\! & 0 & \hbox{in }\overline{B_{R_1}},\vspace{3pt}\\
0\ <\ \varphi & \!\le\! & m & \hbox{in }\overline{B_{R_1}},\vspace{3pt}\\
\varphi & \!=\! & m & \hbox{on }\partial B_{R_1},\vspace{3pt}\\
\displaystyle\mathop{\min}_{\overline{B_{R_1}}}\varphi\ =\ \varphi(0) & \!<\! & \eta. &\eaa\right.
\ee
Furthermore, it follows from~\cite{GNN1} that $\varphi$ is radially symmetric, namely there is a $C^2([0,R_1])$ function $\widetilde\varphi$ such that
\be\label{deftildevarphi}
\varphi(x)=\widetilde\varphi(|x|)\hbox{ for all $x\in\overline{B_{R_1}}$}
\ee
and the Hopf lemma (or, here, the Cauchy-Lipschitz theorem) implies that
\be\label{defdelta}
\delta:=\widetilde\varphi'(R_1)>0.
\ee

Since $u$ is bounded by assumption, it follows from standard parabolic estimates that there is a positive constant $M_2$ such that
\be\label{defM2}
|u_{x_i,x_j}(t,x)|\le M_2\ \hbox{ for all }(t,x)\in\R\times\R^N\hbox{ and }1\le i,j\le N.
\ee
From Lemma~\ref{lem3} applied with $\epsilon=\delta^2/(4M_2)>0$, there is a real number
\be\label{defR2}
R_2=\rho_\epsilon=\rho_{\delta^2/(4M_2)}>0
\ee
such that, if $u(t_0,\cdot)\ge m+\delta^2/(4M_2)$ in $\overline{B(y_0,R_2)}$ for some $(t_0,y_0)\in\R\times\R^N$, then $M<+\infty$ and $\max_{|x|\le\gamma t}|u(t,x)-M|\to0$ as $t\to+\infty$ for some $\gamma>0$. Here, by choosing $M_2$ large if necessary, we may assume that $\ep=\delta^2/(4M_2)<M-m$.

Remember now that $x_0\in\R^N$ is a center of symmetry given by~\eqref{defx0} and that $u$ is localized for $t\le0$, in the sense of~\eqref{localized}. There is then a point $x_1\in\R^N$ such that
$$|x_1-x_0|\ge R_2+\frac{\delta}{2M_2}+R_1\ \hbox{ and }\ u(t,\cdot)<\varphi(0)\hbox{ in }\overline{B(x_1,R_1)}\hbox{ for all }t\le 0.$$
We shall then compare $u$ with $\varphi(\cdot-x_1)$ in $\overline{B(x_1,R_1)}$. First of all, owing to~\eqref{eqvarphi3}, one has $u(t,\cdot)<\varphi(\cdot-x_1)$ in $\overline{B(x_1,R_1)}$ for all $t\le 0$. Two cases may then occur:
\be\label{case1}
\hbox{either }u(t,\cdot)<\varphi(\cdot-x_1)\hbox{ in $\overline{B(x_1,R_1)}$ for all $t\in\R$},
\ee
\be\label{case2}\baa{l}
\hbox{or there is }t_0\in\R\hbox{ such that }u(t,\cdot)<\varphi(\cdot-x_1)\hbox{ in $\overline{B(x_1,R_1)}$ for all $t<t_0$}\vspace{3pt}\\
\qquad\hbox{and $u(t_0,\cdot)\le\varphi(\cdot-x_1)$ in $\overline{B(x_1,R_1)}$ with equality somewhere in $\overline{B(x_1,R_1)}$}.\eaa
\ee
It will turn out that~\eqref{case1} will lead to the conclusions~(i) or~(ii) of Theorem~\ref{th1}, whereas~\eqref{case2} will lead to the spreading case~(iii). We consider in Step~3 the alternative~\eqref{case1}, while~\eqref{case2} will be dealt with in Steps~4 and~5.\hfill\break

\noindent{\it{Step 3: convergence at large times if $u$ is uniformly localized.}} We assume here that~\eqref{case1} holds. Thus, $u(t,x_1)<\varphi(0)<\eta$ for all $t\in\R$ and property~\eqref{defx0} implies that $u(t,x)<\varphi(0)<\eta$ for all $(t,x)\in\R\times\R^N$ with $|x-x_0|\ge|x_1-x_0|$. The arguments used in Remark~\ref{rem15} then yield
$$u(t,x)\to0\hbox{ as $|x|\to+\infty$ uniformly in $t\in\R$},$$
that is, $u$ is uniformly localized. From~\cite[Theorem~1.1]{BJP} (see also~\cite{FP}) and standard para\-bolic estimates, it follows that either $u(t,\cdot)\to0$ as $t\to+\infty$ in $H^1(\R^N)\cap C^2(\R^N)$ (that is, the alternative~(i) holds in Theorem~\ref{th1}), or there is positive steady state $\phi\in\mathcal{E}$ solving~\eqref{steady} such that $u(t,\cdot)\to\phi$ as $t\to+\infty$ in $H^1(\R^N)\cap C^2(\R^N)$ (that is, the alternative~(ii) holds in Theorem~\ref{th1}). Notice that, in the former case, the action $E[u(t,\cdot)]$ defined by~\eqref{defE} satisfies~$E[u(t,\cdot)]\to E[0]=0$ as $t\to+\infty$, while in the latter case,
\be\label{Ephi}
E[u(t,\cdot)]\to E[\phi]\hbox{ as~$t\to+\infty$}.
\ee
In all cases, the function $t\mapsto E[u(t,\cdot)]$ is then bounded in $\R$.\hfill\break

\noindent{\it{Step 4: the transition is radially bounded if $u$ spreads, proof of~\eqref{defxit}.}} We assume in the sequel that~\eqref{case2} holds. We shall see that this case leads to the alternative~(iii) of the conclusion of Theorem~\ref{th1}. We prove the property~\eqref{defxit} in the present Step~4, and property~\eqref{spreading} in Step~5. The proof of~\eqref{defxit} is based on the maximum principle and on suitable estimates on the oscillations of the radial positions of the level sets of $u$ at large time, as well as on the key-lemmas of Step~1.

First of all, since $\varphi$ solves~\eqref{eqvarphi3}, the alternative~\eqref{case2} and the parabolic strong maximum principle imply that there is a point $x_2\in\partial B(x_1,R_1)$ such that
$$u(t_0,x_2)=\varphi(x_2-x_1)=m.$$
Because of~\eqref{defx0},~\eqref{case2} and of the inequality $|x_1-x_0|\ge R_2+\delta/(2M_2)+R_1>R_1$, together with the fact that $\varphi<m$ in $B_{R_1}$ (from the elliptic strong maximum principle), it turns out that~$x_2$ is the unique point lying at the intersection of the sphere $\partial B(x_1,R_1)$ and the segment~$[x_0,x_1]$. In particular,
$$|x_2-x_0|=|x_1-x_0|-R_1\ge R_2+\frac{\delta}{2M_2}.$$
Furthermore, from ~\eqref{case2} and the definitions of $\widetilde{\varphi}$ and $\delta$ satisfying \eqref{deftildevarphi}-\eqref{defdelta}, it follows that
$$-|\nabla u(t_0,x_2)|=\nabla u(t_0,x_2)\cdot\frac{x_2-x_0}{|x_2-x_0|}\le-\delta<0.$$
Together with~\eqref{defM2} and~\eqref{defx0} again, one infers that
$$\nabla u(t_0,x)\cdot\frac{x-x_0}{|x-x_0|}\le-\frac{\delta}{2}<0\ \hbox{ for all }x\hbox{ such that }|x_2-x_0|-\frac{\delta}{2M_2}\le|x-x_0|\le|x_2-x_0|.$$
Since $u(t_0,\cdot)=m$ at the point $x_2$ and then on $\partial B(x_0,|x_2-x_0|)$, it follows that
$$u(t_0,\cdot)\ge m+\frac{\delta^2}{4M_2}\hbox{ on }\partial B\big(x_0,|x_2-x_0|-\frac{\delta}{2M_2}\big),\hbox{ with }|x_2-x_0|-\frac{\delta}{2M_2}\ge R_2,$$
hence
\be\label{inequmeps}
u(t_0,\cdot)\ge m+\frac{\delta^2}{4M_2}\hbox{ in }\overline{B(x_0,R_2)}
\ee
by~\eqref{defx0} again. Lemma~\ref{lem3} and the definition~\eqref{defR2} of $R_2$ then imply that
$$M<+\infty$$
and
\be\label{spreading3}
\max_{|x|\le\gamma t}|u(t,x)-M|\to0\hbox{ as $t\to+\infty$}
\ee
for some $\gamma>0$.

Secondly,~\eqref{defx0} together with~\eqref{expdecay} and~\eqref{spreading3} yield the existence of a real number $\tau_1$ such that
\be\label{deftau1}
\hbox{for each }t\ge\tau_1,\left\{\baa{l}
\displaystyle\max_{\R^N}u(t,\cdot)=u(t,x_0)>m\vspace{3pt}\\
\hbox{there is a unique }\xi(t)>0\hbox{ such that $u(t,\cdot)=m\hbox{ on }\partial B(x_0,\xi(t))$},\eaa\right.
\ee
and $\liminf_{t\to+\infty}\xi(t)/t\ge\gamma>0$. In particular,
\be\label{xiinfty}
\lim_{t\to+\infty}\xi(t)=+\infty.
\ee
The implicit function theorem with~\eqref{defx0} implies that the function $t\mapsto\xi(t)$ is of class $C^1([\tau_1,+\infty))$. Define also $\xi(t)=\xi(\tau_1)$ for all $t<\tau_1$. The function $\xi$ is then continuous in $\R$. We shall show in this Step~4 that~\eqref{defxit} holds with this function $\xi$. To do so, we first prove in the following two lemmas some key-properties on the local oscillations of the function $\xi$.

\begin{lemma}\label{lem5}
There is a positive constant $\tau_2$ such that $\xi(t+s)>\xi(t)$ for all $t\ge\tau_1$ and $s\ge\tau_2$.
\end{lemma}  

\begin{proof}
By~\eqref{localized} and~\eqref{expdecay}, there is a real number $R_3>0$ such that
\be\label{defR3}
u(t,x)<\varphi(0)\hbox{ for all $t\le\tau_1$ and $x$ such that $|x-x_0|\ge R_3$}.
\ee
With $\epsilon=\delta^2/(4M_2)>0$, $\lambda=(m+M)/2<M$ and $r=\delta/(2M_2)\ge0$, denote,
using the notations of Lemma~\ref{lem4} and the definition~\eqref{defR2} of $R_2$,
\be\label{defR4}
R_4=\max\Big(\rho_{\epsilon,\lambda,r}+\frac{\delta}{2M_2},R_3,\xi(\tau_1)+1\Big)>0
\ee
and
$$\tau_2=T_{\epsilon,\lambda,r}>0.$$
Let also $\tau_3\in\R$ be such that
$$\xi(t)\ge R_4\ \hbox{ for all $t\ge\tau_3$}$$
(hence, $\tau_3>\tau_1$, since $R_4>\xi(\tau_1)$).

Consider now any $t\ge\tau_3$ and $s\ge\tau_2$ and let us show that $\xi(t+s)>\xi(t)$. Let~$x_3\in\R^N$ be such that
$$|x_3-x_0|=\xi(t)+R_1,$$
where $R_1>0$ is given in~\eqref{eqvarphi3}. Thus, $|x_3-x_0|\ge R_4+R_1\ge R_3+R_1$, hence $u(t',\cdot)<\varphi(0)\le\varphi(\cdot-x_3)$ in $\overline{B(x_3,R_1)}$ for all $t'\le\tau_1$ by~\eqref{defR3}. Observe that $\xi(t')=\xi(\tau_1)<R_4\le\xi(t)$ for all $t'\le\tau_1$, and, by continuity of $\xi$, denote
$$t^*=\min\big\{t'\in(-\infty,t]: \xi(t')=\xi(t)\big\}\in(\tau_1,t].$$
Let $x_4$ be the intersection point of the segment $[x_0,x_3]$ with $\partial B(x_3,R_1)$. One has
$$|x_4-x_0|=|x_3-x_0|-R_1=\xi(t)=\xi(t^*), $$
hence $u(t^*,x_4)=m=\varphi(x_4-x_3)$. Furthermore, $u(t',\cdot)<\varphi(0)\le\varphi(\cdot-x_3)$ in $\overline{B(x_3,R_1)}$ for all $t'\le\tau_1$ by~\eqref{eqvarphi3} and~\eqref{defR3}, while
$$u(t',\cdot)\le m=\varphi(\cdot-x_3)\hbox{ on $\partial B(x_3,R_1)$ for all~$t'\in[\tau_1,t^*]$}$$
by~\eqref{defx0} and the definition of~$t^*$ (and even $u(t',\cdot)<m$ on $\partial B(x_3,R_1)$ for all~$t'\in[\tau_1,t^*)$). It then follows from the maximum principle that
$$u(t',\cdot)\le\varphi(\cdot-x_3)\hbox{ in $\overline{B(x_3,R_1)}$ for all $t'\in[\tau_1,t^*]$}$$
(actually with strict inequality for $t'\in[\tau_1,t^*)$ and even for $t'\in(-\infty,t^*)$). In particular, $u(t^*,\cdot)\le\varphi(\cdot-x_3)$ in $\overline{B(x_3,R_1)}$ and since $x_4\in\partial B(x_3,R_1)$ with $|x_4-x_0|=\xi(t^*)$ and $t^*\ge\tau_1$, one has $u(t^*,x_4)=m=\varphi(x_4-x_3)$. Therefore,
$$-|\nabla u(t^*,x_4)|=\nabla u(t^*,x_4)\cdot\frac{x_4-x_0}{|x_4-x_0|}\le-\delta$$
owing to the definition of $\delta$ in \eqref{deftildevarphi}-\eqref{defdelta}. Hence, as in the proof of~\eqref{inequmeps}, one infers that
\be\label{ut*}
u(t^*,\cdot)\ge m+\frac{\delta^2}{4M_2}=m+\epsilon\hbox{ in }\overline{B\big(x_0,|x_4-x_0|-\delta/(2M_2)\big)},
\ee
with $|x_4-x_0|-\delta/(2M_2)=\xi(t)-\delta/(2M_2)\ge R_4-\delta/(2M_2)\ge\rho_{\epsilon,\lambda,r}$ by~\eqref{defR4}. Lemma~\ref{lem4} then yields
$$u(t',\cdot)\ge\lambda=\frac{m+M}{2}\hbox{ in }\overline{B\big(x_0,|x_4-x_0|-\delta/(2M_2)+r\big)}=\overline{B(x_0,|x_4-x_0|)}=\overline{B(x_0,\xi(t))}$$
for all $t'\ge t^*+T_{\epsilon,\lambda,r}=t^*+\tau_2$. Since $t+s\ge t^*+\tau_2$, one has $u(t+s,\cdot)\ge(m+M)/2>m$ in $\overline{B(x_0,\xi(t))}$ and the definition of $\xi(t+s)$ together with~\eqref{defx0} and $t+s>t^*>\tau_1$ finally yields $\xi(t+s)>\xi(t)$.

As a consequence, $\xi(t+s)>\xi(t)$ for all $t\ge\tau_3$ and $s\ge\tau_2$. Since $\xi$ is continuous in~$\R$ and~$\xi(t)\to+\infty$ as $t\to+\infty$, the conclusion of Lemma~\ref{lem5} follows, even if it means increa\-sing~$\tau_2$ if necessary.
\end{proof}

\begin{lemma}\label{lem6}
For each $\tau>0$, there is a positive constant $A_\tau$ such that $\xi(t+s)\le\xi(t)+A_\tau$ for all $t\in\R$  and $s\in[0,\tau]$.
\end{lemma}  

\begin{proof}
Assume that the conclusion does not hold. Then there are $\tau>0$ and some sequences~$(t_n)_{n\in\N}$ in $\R$ and $(s_n)_{n\in\N}$ in $[0,\tau]$ such that $\xi(t_n+s_n)>\xi(t_n)+n$. Since $\xi$ is conti\-nuous in $\R$ and constant in $(-\infty,\tau_1]$, it follows that $t_n\to+\infty$ as $n\to+\infty$, hence~$\xi(t_n)\to+\infty$ as~$n\to+\infty$ by~\eqref{xiinfty}. Without loss of generality, one can assume that, for every $n\in\N$,
$$t_n\ge\tau_1+\tau_2\ \hbox{ and }\ \xi(t_n)\ge\max(R_3,\xi(\tau_1)+1),$$
where $\tau_1\in\R$ and $\tau_2>0$ are given in~\eqref{deftau1} and in Lemma~\ref{lem5}, and $R_3>0$ is given in~\eqref{defR3}. 

Now, for every $n\in\N$, Lemma~\ref{lem5} yields the existence of $t^*_n\in(t_n-\tau_2,t_n]\,(\subset(\tau_1,t_n])$ such that $\xi(t^*_n)=\xi(t_n)$ and $\xi(t)<\xi(t^*_n)=\xi(t_n)$ for all $t<t^*_n$. Let $y_n\in\R^N$ be such that
$$|y_n-x_0|=\xi(t^*_n)+R_1=\xi(t_n)+R_1\,(\ge R_3+R_1).$$
Since $u(t,\cdot)<\varphi(0)\le\varphi(\cdot-y_n)$ in $\overline{B(y_n,R_1)}$ for all $t\le\tau_1$ by~\eqref{defR3}, and since $u(t,\cdot)<m=\varphi(\cdot-y_n)$ on $\partial B(y_n,R_1)$ for all $t\in[\tau_1,t^*_n)$ by~\eqref{defx0} and definition of $t^*_n$, the maximum principle implies that
\be\label{t*nyn}
u(t^*_n,\cdot)\le\varphi(\cdot-y_n)\hbox{ in $\overline{B(y_n,R_1)}$}.
\ee
In particular, $u(t^*_n,y_n)\le\varphi(0)$ and
\be\label{t*n}
u(t^*_n,x)\le\varphi(0)\hbox{ for all $x$ such that $|x-x_0|\ge|y_n-x_0|=\xi(t_n)+R_1$},
\ee
by~\eqref{defx0}. 

On the other hand, for every $n\in\N$, one has $t_n+s_n\ge t_n\ge\tau_1+\tau_2>\tau_1$, and there is a point $z_n$ such that $|z_n-x_0|=\xi(t_n+s_n)$, hence $u(t_n+s_n,z_n)=m$. Notice also that~$t_n-t^*_n+s_n\in[0,\tau_2+\tau)$ for each $n\in\N$. Up to extraction of a subsequence, one can assume without loss of generality that~$t_n-t^*_n+s_n\to s_\infty\in[0,\tau_2+\tau]$ as $n\to+\infty$ and that the functions
$$u_n:(t,x)\mapsto u_n(t,x)=u(t+t^*_n,x+z_n)$$
converge in $C^{1,2}_{t,x}$ locally in $\R\times\R^N$ to a bounded nonnegative solution $u_\infty$ of~\eqref{eq:ACN} such that $u_\infty(s_\infty,0)=m$. Furthermore, for each $x\in\R^N$, there holds
$$|x+z_n-x_0|\ge|z_n-x_0|-|x|=\xi(t_n+s_n)-|x|>\xi(t_n)+n-|x|,$$
hence $|x+z_n-x_0|\ge\xi(t_n)+R_1$ for all $n$ large enough and $u_n(0,x)=u(t^*_n,x+z_n)\le\varphi(0)$ by~\eqref{t*n}. As a consequence, $u_\infty(0,x)\le\varphi(0)$ for all $x\in\R^N$. Since $0<\varphi(0)<\eta$ by~\eqref{eqvarphi3} and $f<0$ in $(0,\eta]$ by~\eqref{defeta}, it follows from the maximum principle that $u_\infty\le\varphi(0)<\eta$ in~$[0,+\infty)\times\R^N$. In particular, $u_\infty(s_\infty,0)<\eta$, which is impossible since $u_\infty(s_\infty,0)=m$ and~$m>\eta$ (remember that $f(m)=0$ and $m>0$). One has then reached a contradiction, and the proof of Lemma~\ref{lem6} is thereby complete.
\end{proof}

With Lemmas~\ref{lem5} and~\ref{lem6} in hand, we can now complete the proof of~\eqref{defxit}. Let us begin with the first statement in~\eqref{defxit}. Assume by way of contradiction that it does not hold. Then, thanks to~\eqref{0uM}, there are $M'\in(0,M)$ and some sequences $(t_n)_{n\in\N}$ converging to $+\infty$ and~$(x_n)_{n\in\N}$ in $\R^N$ such that
\be\label{defM'}
0<u(t_n,x_n)\le M'<M\hbox{ for all }n\in\N,\ \hbox{ and }|x_n|-\xi(t_n)\to-\infty\hbox{ as }n\to+\infty.
\ee

Let $\tau>0$ be an arbitrary positive real number, and let $R_1>0$ be given as in~\eqref{eqvarphi3}. Consider in this paragraph the indices $n$ large enough so that $t_n-\tau\ge\tau_1+\tau_2$ for every~$n\in\N$, where $\tau_1\in\R$ and $\tau_2>0$ are given in~\eqref{deftau1} and Lemma~\ref{lem5}, and
$$\xi(t_n-\tau)\ge\max\big(R_3,\xi(\tau_1)+1,\frac{\delta}{2M_2}+1\big),$$
where $R_3>0$ is given in~\eqref{defR3}, $\delta>0$ in~\eqref{deftildevarphi}-\eqref{defdelta} and $M_2>0$ in~\eqref{defM2}. Notice that the quantities $\tau$, $\tau_1$, $\tau_2$, $R_1$, $R_3$, $\delta$ and $M_2$ are independent of $n$. Now, for each $n$ large enough, Lemma~\ref{lem5} yields the existence of $t^*_n\in(t_n-\tau-\tau_2,t_n-\tau]\,(\subset(\tau_1,t_n-\tau])$ such that~$\xi(t^*_n)=\xi(t_n-\tau)$ and $\xi(t)<\xi(t^*_n)=\xi(t_n-\tau)$ for all $t<t^*_n$. Let $y_n\in\R^N$ be such that
$$|y_n-x_0|=\xi(t^*_n)+R_1=\xi(t_n-\tau)+R_1\,(\ge R_3+R_1)$$
and $z_n$ be the intersection point of $[x_0,y_n]$ with $\partial B(y_n,R_1)$ such that
$$|z_n-x_0|=|y_n-x_0|-R_1=\xi(t_n-\tau)=\xi(t^*_n),\quad u(t^*_n,z_n)=m$$
by~\eqref{deftau1}. As in the proof of~\eqref{t*nyn}, there holds $u(t^*_n,\cdot)\le\varphi(\cdot-y_n)$ in~$\overline{B(y_n,R_1)}$ with $u(t^*_n,z_n)=m=\varphi(z_n-y_n)$. Therefore,
$$-|\nabla u(t^*_n,z_n)|=\nabla u(t^*_n,z_n)\cdot\frac{z_n-x_0}{|z_n-x_0|}\le-\delta,$$
with $\delta>0$ given by~\eqref{deftildevarphi}-\eqref{defdelta}. Hence, as in the proof of~\eqref{ut*}, one infers that
\be\label{ut*bis}
u(t^*_n,\cdot)\ge m+\frac{\delta^2}{4M_2}\ \hbox{ in }\overline{B\big(x_0,|z_n-x_0|-\delta/(2M_2)\big)}=\overline{B\big(x_0,\xi(t_n-\tau)-\delta/(2M_2)\big)},
\ee
with $|z_n-x_0|-\delta/(2M_2)=\xi(t_n-\tau)-\delta/(2M_2)>0$. Together with~\eqref{0uM}, this implies in particular that $m<m+\delta^2/(4M_2)<M$. Notice also that $\tau\le t_n-t^*_n<\tau+\tau_2$ for each~$n$ (large enough). Up to extraction of a subsequence, one has $t_n-t^*_n\to s_\infty\in[\tau,\tau+\tau_2]$ as~$n\to+\infty$ and the functions
$$u_n:(t,x)\mapsto u_n(t,x)=u(t+t^*_n,x+x_n)$$
converge in $C^{1,2}_{t,x}$ locally in $\R\times\R^N$ to a bounded nonnegative solution $u_\infty$ of~\eqref{eq:ACN} such that
$$u_\infty(s_\infty,0)\le M'<M$$
by~\eqref{defM'}. Furthermore, for each $x\in\R^N$, one has
$$|x+x_n-x_0|\le\xi(t_n)+|x_n|-\xi(t_n)+|x-x_0|\le\xi(t_n-\tau)+A_\tau+|x_n|-\xi(t_n)+|x-x_0|$$
from Lemma~\ref{lem6}, where $A_\tau>0$ is given in Lemma~\ref{lem6}, hence $|x+x_n-x_0|\le\xi(t_n-\tau)-\delta/(2M_2)$ for all $n$ large enough, from the second statement of~\eqref{defM'}. As a consequence, $u_n(0,x)=u(t^*_n,x+x_n)\ge m+\delta^2/(4M_2)$ for all $n$ large enough, by~\eqref{ut*bis}. Thus, $u_\infty(0,\cdot)\ge m+\delta^2/(4M_2)$ in $\R^N$ and $u_\infty(t,\cdot)\ge\varpi(t)$ in $\R^N$ for all $t\ge0$, where $\varpi$ obeys
$$\left\{\baa{ll}
\varpi'(t)=f(\varpi(t)) & \hbox{ for $t\ge0$},\vspace{3pt}\\
\displaystyle\varpi(0)=m+\frac{\delta^2}{4M_2}>m. & \eaa\right.$$
Since $m<m+\delta^2/(4M_2)<M$ and $f>0$ in $(m,M)$ with $f(M)=0$, the function $\varpi$ is increasing in $[0,+\infty)$ and $\varpi(+\infty)=M$. The inequality $u_\infty(t,\cdot)\ge\varpi(t)$ applied at $t=s_\infty\ge\tau>0$ and $x=0$ yields $u_\infty(s_\infty,0)\ge\varpi(s_\infty)\ge\varpi(\tau)$, hence $M>M'\ge u_\infty(s_\infty,0)\ge\varpi(\tau)$. Since~$M'$ is given in~\eqref{defM'} independently of $\tau>0$ and $\tau>0$ can be arbitrarily large, one infers that $M>M'\ge\varpi(+\infty)=M$, a contradiction. As a consequence, the first line in~\eqref{defxit} has been proved.

Let us now show the second statement in~\eqref{defxit}. Assume by way of contradiction that it does not hold. Then, thanks to~\eqref{0uM}, there are $\kappa\in(0,M)$ and some sequences $(t_n)_{n\in\N}$ converging to $+\infty$ and~$(x_n)_{n\in\N}$ in $\R^N$ such that
\be\label{defkappa}
0<\kappa\le u(t_n,x_n)<M\hbox{ for all }n\in\N,\ \hbox{ and }|x_n|-\xi(t_n)\to+\infty\hbox{ as }n\to+\infty.
\ee
Let $\sigma>0$ be an arbitrary positive real number such that
$$\sigma\ge\tau_2,$$
where $\tau_2>0$ is given in Lemma~\ref{lem5}, and let $R_1>0$ be given as in~\eqref{eqvarphi3}. Consider in this paragraph the indices $n$ large enough so that $t_n-\sigma\ge\tau_1+\tau_2$ for every~$n\in\N$, where $\tau_1\in\R$ is given in~\eqref{deftau1}, and
$$\xi(t_n-\sigma)\ge\max(R_3,\xi(\tau_1)+1),$$
where $R_3>0$ is given in~\eqref{defR3}. For each $n$ large enough, Lemma~\ref{lem5} yields the existence of~$t^*_n\in(t_n-\sigma-\tau_2,t_n-\sigma]\,(\subset(\tau_1,t_n-\sigma])$ such that~$\xi(t^*_n)=\xi(t_n-\sigma)$ and $\xi(t)<\xi(t^*_n)=\xi(t_n-\sigma)$ for all $t<t^*_n$. Let $y_n\in\R^N$ satisfy
$$|y_n-x_0|=\xi(t^*_n)+R_1=\xi(t_n-\sigma)+R_1\,(\ge R_3+R_1).$$
As for~\eqref{t*nyn}, one then has $u(t^*_n,\cdot)\le\varphi(\cdot-y_n)$ in~$\overline{B(y_n,R_1)}$. In particular,~$u(t^*_n,y_n)\le\varphi(0)$ and, from~\eqref{defx0},
\be\label{t*nbis}
u(t^*_n,x)\le\varphi(0)\ \hbox{ for all $x$ such that $|x-x_0|\ge|y_n-x_0|=\xi(t_n-\sigma)+R_1$}.
\ee
Notice also that $\sigma\le t_n-t^*_n<\sigma+\tau_2$ for each~$n$ (large enough). Up to extraction of a subsequence, one has $t_n-t^*_n\to t_\infty\in[\sigma,\sigma+\tau_2]$ as~$n\to+\infty$ and the functions
$$u_n:(t,x)\mapsto u_n(t,x)=u(t+t^*_n,x+x_n)$$
converge in $C^{1,2}_{t,x}$ locally in $\R\times\R^N$ to a bounded nonnegative solution $u_\infty$ of~\eqref{eq:ACN} such that
$$0<\kappa\le u_\infty(t_\infty,0)$$
by~\eqref{defkappa}. Furthermore, for each $x\in\R^N$, one has
$$|x+x_n-x_0|\ge|x_n|-\xi(t_n)+\xi(t_n)-|x-x_0|>|x_n|-\xi(t_n)+\xi(t_n-\sigma)-|x-x_0|$$
from Lemma~\ref{lem5}, since $t_n-\sigma\ge\tau_1$ and $\sigma\ge\tau_2$. Hence $|x+x_n-x_0|\ge\xi(t_n-\sigma)+R_1$ for all $n$ large enough, from~\eqref{defkappa}. As a consequence, $u_n(0,x)=u(t^*_n,x+x_n)\le\varphi(0)$ for all $n$ large enough, by~\eqref{t*nbis}. Thus, $u_\infty(0,\cdot)\le\varphi(0)$ in $\R^N$ and $u_\infty(t,\cdot)\le\vartheta(t)$ in $\R^N$ for all $t\ge0$, where $\vartheta$ obeys
$$\left\{\baa{ll}
\vartheta'(t)=f(\vartheta(t)) & \hbox{ for $t\ge0$},\vspace{3pt}\\
\displaystyle\vartheta(0)=\varphi(0)\in(0,\eta). & \eaa\right.$$
Since $f<0$ in $(0,\eta)$ with $f(0)=0$, the function $\vartheta$ is decreasing in $[0,+\infty)$ and $\vartheta(+\infty)=0$. The inequality $u_\infty(t,\cdot)\le\vartheta(t)$ applied at $t=t_\infty\ge\sigma>0$ and $x=0$ yields $u_\infty(t_\infty,0)\le\vartheta(t_\infty)\le\vartheta(\sigma)$, hence $0<\kappa\le\vartheta(\sigma)$. Since~$\kappa$ is given in~\eqref{defkappa} independently of $\sigma$ and since~$\sigma\ge\tau_2$ can be arbitrarily large, one infers that $0<\kappa\le\vartheta(+\infty)=0$, a contradiction. As a conclusion, the proof of~\eqref{defxit} is thereby complete.

\begin{remark}
The quantities~$\xi(t)$ given in~\eqref{deftau1} (for $t\ge\tau_1$) are the radial positions (with respect to the point $x_0$) of the level sets with level $m$. Now, for any level $\lambda$ in $(0,M)$, there is a unique real number $\xi_\lambda(t)>0$ such that $u(t,x)=\lambda$ if and only if $|x-x_0|=\xi_\lambda(t)$, for all~$t$ large enough. One then infers from~\eqref{defxit} that
$$\limsup_{t\to+\infty}|\xi_\lambda(t)-\xi(t)|<+\infty.$$
Furthermore, it also follows from~\eqref{defx0} that, for any unit vector $e$ and any sequence $(t_n)_{n\in\N}$ converging to $+\infty$, the functions
$$u_n:(t,x)\mapsto u_n(t,x)=u(t+t_n,x+\xi_\lambda(t_n)e)$$
converge in $C^{1,2}_{t,x}$ locally in $\R\times\R^N$, up to extraction of a subsequence, to a bounded nonnegative solution $u_\infty$ of~\eqref{eq:ACN} which only depends on $t$ and the variable $x\cdot e$ and is nonincreasing in the direction $e$. Moreover,~\eqref{defxit} implies that $u_\infty(0,x)\to M$ as $x\cdot e\to-\infty$ and $u_\infty(0,x)\to 0$ as $x\cdot e\to+\infty$. In particular, the nonpositive function $e\cdot\nabla u_\infty$ can not be identically~$0$ in~$(-\infty,0]\times\R^N$ and the strong parabolic maximum principle then yields $e\cdot\nabla u_\infty(0,0)<0$. From the arbitrariness of the sequence $(t_n)_{n\in\N}$ converging to $+\infty$ and from~\eqref{defx0}, we conclude that
$$\liminf_{t\to+\infty,\,u(t,x)=\lambda}|\nabla u(t,x)|>0,\hbox{ that is, }\limsup_{t\to+\infty,\,u(t,x)=\lambda}\nabla u(t,x)\cdot\frac{x-x_0}{|x-x_0|}<0,$$
for any level $\lambda\in(0,M)$. In other words, the radial derivatives of the function $u$ do not degenerate at large times along any level set of $u$.
\end{remark}

\noindent{\it{Step 5: asymptotic position of the level sets if $u$ spreads, proof of~\eqref{spreading}.}} We still assume that~\eqref{case2} holds, hence $M<+\infty$ and~\eqref{defxit} holds, together with~\eqref{spreading3}. We shall show here that $\xi(t)/t$ has a well determined limit as $t\to+\infty$. Such a property has been well known since the seminal paper~\cite{AW}, under some additional assumptions on the function~$f$. The proof given in~\cite{AW} was based on some comparison arguments and on the existence of approximated fronts defined in bounded intervals or in half-lines. The proof of~\eqref{spreading} given here is still based on comparison arguments with suitable sub- and supersolutions, but the approximated fronts moving at speeds arbitrarily close to $c$ that are here used are defined in the whole real line and are given in Lemma~\ref{lem1}.

First of all, as already emphasized, one has $0<m<M$, $f(0)=f(M)=0$, $f'(0)<0$, $f>0$ in $(m,M)$, $F<0$ in $(0,m]$ and $F(M)>0$. Lemma~\ref{lem1} then yields the existence and uniqueness of a pair $(c,\varphi)$ solving~\eqref{eqvarphi}, with $c>0$.

Consider now any $c'\in(0,c)$, and let us show that $\liminf_{t\to+\infty}\xi(t)/t\ge c'$. From Remark~\ref{rem32}, there is a sequence $(\epsilon_n)_{n\in\N}$ in $(0,m)$ converging to $0$ such that, for each $n\in\N$, there is a $C^1([-\epsilon_n,M-\epsilon_n])$ function $\underline{f}_n$ such that $\underline{f}_n(-\epsilon_n)=\underline{f}_n(M-\epsilon_n)=0$, $\underline{f}_n\le f$ in~$[0,M-\epsilon_n]$, and there is a unique pair $(c_n,\varphi_n)\in\R\times C^2(\R)$ solving
\be\label{varphinbis}
\varphi_n''+c_n\varphi_n'+\underline{f}_n(\varphi_n)=0\hbox{ in }\R,\ \varphi_n'<0\hbox{ in }\R,\ \varphi_n(-\infty)=M-\epsilon_n,\ \varphi_n(+\infty)=-\epsilon_n.
\ee
Furthermore, $c_n<c$ and $c_n\to c$ as $n\to+\infty$. Fix $\epsilon>0$ arbitrary, and then $n$ large enough such that
$$0<\epsilon_n\le\epsilon\hbox{ and }c'<c_n<c.$$
Let then $\rho>0$ large enough such that
$$\frac{N-1}{\rho}<\frac{c_n-c'}{2}$$
and denote $c'_n=(c'+c_n)/2\in(c',c_n)\subset(c',c)$. Since $u(t,\cdot)\to M$ as $t\to+\infty$ locally uniformly in $\R^N$ by~\eqref{spreading3}, there is a time $T>0$ such that
\be\label{uMepsn}
u(t,x)\ge M-\epsilon_n\ \hbox{ for all }t\ge T\hbox{ and }|x-x_0|\le\rho.
\ee
Let then $A>0$ be such that $\varphi_n(r-c'_nT+A)<0$ for all $r\ge\rho$ (that is possible since $\varphi_n(+\infty)=-\epsilon_n<0$). Let us finally define
$$\underline{u}(t,x):=\max\big(\varphi_n(|x-x_0|-c'_nt+A),0\big)$$
and show that this function is a generalized subsolution of~\eqref{eq:ACN} for $t\ge T$ and $|x-x_0|\ge\rho$. First of all, at time $t=T$, for all $|x-x_0|\ge\rho$, one has $\varphi_n(|x-x_0|-c'_nT+A)<0$, hence $\underline{u}(T,x)=0<u(T,x)$. Furthermore, for all $t\ge T$ and $|x-x_0|=\rho$, one has $\underline{u}(t,x)<M-\epsilon_n\le u(t,x)$. Since $f(0)=0$, it just remains to show that, for any $(t,x)$ such that $t>T$ and $|x-x_0|>\rho$ with $\underline{u}(t,x)>0$, then $\underline{u}_t(t,x)\le\Delta\underline{u}(t,x)+f(\underline{u}(t,x))$. Pick any such $(t,x)$ and notice that $\underline{u}(t,x)=\varphi_n(|x-x_0|-c'_nt+A)\in(0,M-\epsilon_n)$ in a neighborhood of $(t,x)$. Hence, having~\eqref{varphinbis} in mind, it follows that
$$\baa{l}
\underline{u}_t(t,x)-\Delta\underline{u}(t,x)-f(\underline{u}(t,x))\vspace{3pt}\\
\qquad\qquad\qquad\displaystyle=-c'_n\varphi_n'(|x-x_0|-c'_nt+A)-\varphi_n''(|x-x_0|-c'_nt+A)\vspace{3pt}\\
\qquad\qquad\qquad\displaystyle\ \ -\frac{N-1}{|x-x_0|}\varphi_n'(|x-x_0|-c'_nt+A)-f(\varphi_n(|x-x_0|-c'_nt+A))\vspace{3pt}\\
\qquad\qquad\qquad\displaystyle\le\Big(c_n-c'_n-\frac{N-1}{|x-x_0|}\Big)\varphi_n'(|x-x_0|-c'_nt+A)\vspace{3pt}\\
\qquad\qquad\qquad<0\eaa$$
since $\underline{f}_n\le f$ in $[0,M-\epsilon_n]$, $(N-1)/|x-x_0|\le(N-1)/\rho<(c_n-c')/2=c_n-c'_n$, and $\varphi'_n<0$ in $\R$. The maximum principle then implies that
$$u(t,x)\ge\underline{u}(t,x)\ge\varphi_n(|x-x_0|-c'_nt+A)\ \hbox{ for all }t\ge T\hbox{ and }|x-x_0|\ge\rho.$$
Therefore, together with~\eqref{uMepsn}, one gets that, for all $t\ge T$,
$$\min_{|x-x_0|\le c't}u(t,x)\ge\varphi_n(c't-c'_nt+A)\to M-\epsilon_n\hbox{ as $t\to+\infty$},$$
since $c'<c'_n$ and~$\varphi_n(-\infty)=M-\epsilon_n$. Together with the inequality $0<u<M$ in $\R\times\R^N$, one infers that~$\limsup_{t\to+\infty}\max_{|x-x_0|\le c't}|u(t,x)-M|\le\epsilon_n\le\epsilon$. Since $\epsilon>0$ and $c'\in(0,c)$ were arbitrary, one concludes from the definition~\eqref{deftau1} of $\xi(t)$ that
\be\label{liminf}
\liminf_{t\to+\infty}\frac{\xi(t)}{t}\ge c.
\ee

For the converse inequality, consider any $c''\!>\!c$ and let us show that $\limsup_{t\to+\infty}\xi(t)/t\!\le\! c''$. From the proof of Lemma~\ref{lem1}, there is a sequence $(\eta_k)_{k\in\N}$ in $(0,m)$ converging to $0$ such that, for each $k\in\N$, there is a $C^1([\eta_k,M+\eta_k])$ function $\overline{f}_k$ such that $\overline{f}_k(\eta_k)=\overline{f}_k(M+\eta_k)=0$, $\overline{f}_k\ge f$ in $[\eta_k,M]$, and there is a unique pair $(\gamma_k,\phi_k)\in\R\times C^2(\R)$ solving
\be\label{phik}
\phi_k''+\gamma_k\phi_k'+\overline{f}_k(\phi_k)=0\hbox{ in }\R,\ \phi_k'<0\hbox{ in }\R,\ \phi_k(-\infty)=M+\eta_k,\ \phi_k(+\infty)=\eta_k.
\ee
Furthermore, $\gamma_k>c$ and $\gamma_k\to c$ as $k\to+\infty$. Fix $\epsilon>0$ arbitrary, and then $k$ large enough such that
$$0<\eta_k\le\epsilon\hbox{ and }c<\gamma_k<c''.$$
Since $0<u<M$ in $\R\times\R^N$ and $u(0,x)\to0$ as $|x|\to+\infty$ by~\eqref{localized}, and since $\phi_k(-\infty)=M+\eta_k>M$ and $\phi_k>\phi_k(+\infty)=\eta_k>0$, there is $A'>0$ such that $\phi_k(1-A')>M$ and~$u(0,x)<\phi_k(|x-x_0|-A')$ for all $x\in\R^N$. Let us finally define
$$\overline{u}(t,x):=\min\big(\phi_k(|x-x_0|-\gamma_kt-A'),M\big)$$
and show that this function is a generalized supersolution of~\eqref{eq:ACN} for $t\ge 0$ and $|x-x_0|\ge1$. First of all, at time $t=0$, one has $u(0,x)<\overline{u}(0,x)$ for all $|x-x_0|\ge1$ (and even for all~$x\in\R^N$, by construction). Furthermore, for all~$t\ge0$ and $|x-x_0|=1$, one has $\gamma_kt\ge0$ and $\phi_k(|x-x_0|-\gamma_kt-A')\ge\phi_k(1-A')>M$, hence $\overline{u}(t,x)=M>u(t,x)$. Since $f(M)=0$, it just remains to show that, for any $(t,x)$ such that $t>0$ and $|x-x_0|>1$ with $\overline{u}(t,x)<M$, then $\overline{u}_t(t,x)\ge\Delta\overline{u}(t,x)+f(\overline{u}(t,x))$. Pick any such $(t,x)$ and notice that $\overline{u}(t,x)=\phi_k(|x-x_0|-\gamma_kt-A')\in(\eta_k,M)$ in a neighborhood of $(t,x)$. Hence, having~\eqref{phik} in mind, it follows that
$$\baa{l}
\overline{u}_t(t,x)-\Delta\overline{u}(t,x)-f(\overline{u}(t,x))\vspace{3pt}\\
\qquad\qquad\qquad\displaystyle=-\gamma_k\phi_k'(|x-x_0|-\gamma_kt-A')-\phi_k''(|x-x_0|-\gamma_kt-A')\vspace{3pt}\\
\qquad\qquad\qquad\displaystyle\ \ -\frac{N-1}{|x-x_0|}\phi_k'(|x-x_0|-\gamma_kt-A')-f(\phi_k(|x-x_0|-\gamma_kt-A'))\vspace{3pt}\\
\qquad\qquad\qquad\displaystyle\ge-\frac{N-1}{|x-x_0|}\phi_k'(|x-x_0|-\gamma_kt-A')\vspace{3pt}\\
\qquad\qquad\qquad\ge0\eaa$$
since $\overline{f}_k\ge f$ in $[\eta_k,M]$ and $\phi'_k<0$ in $\R$. The maximum principle then implies that
$$u(t,x)\le\overline{u}(t,x)\le\phi_k(|x-x_0|-\gamma_kt-A')\ \hbox{ for all }t\ge 0\hbox{ and }|x-x_0|\ge1.$$
Therefore, for all $t\ge 1/c''$, one has $\max_{|x-x_0|\ge c''t}u(t,x)\le\phi_k(c''t-\gamma_kt-A')\to\eta_k$ as $t\to+\infty$, since $c''>\gamma_k$ and~$\phi_k(+\infty)=\eta_k$. One then infers that
$$\limsup_{t\to+\infty}\max_{|x-x_0|\ge c''t}u(t,x)\le\eta_k\le\epsilon.$$
Since~$\epsilon>0$ and~$c''>c$ were arbitrary, one concludes from the definition~\eqref{deftau1} of $\xi(t)$ that $\limsup_{t\to+\infty}\xi(t)/t\le c$. Together with~\eqref{liminf}, the inequality~\eqref{spreading} follows. The proof of Theorem~\ref{th1} is thereby complete.\hfill$\Box$


\section{Proof of Corollaries~\ref{cor1} and~\ref{cor3}}\label{sec5}

As already emphasized in Section~\ref{sec:2}, Corollary~\ref{cor2} follows directly from Theorem~\ref{th1}, while Corollaries~\ref{cor4}-\ref{cor7} follow from Corollaries~\ref{cor2} and~\ref{cor3} and the results of Section~\ref{sec:1}. It just remains to complete the proof of Corollaries~\ref{cor1} and~\ref{cor3}.

\begin{proof}[Proof of Corollary $\ref{cor1}$]
From the observations in the paragraph before Corollary~\ref{cor1} and from the assumptions made in Corollary~\ref{cor1}, there is a unique solution $\phi_0$ of~\eqref{steady} such that $\max_\R\phi_0=\phi_0(0)=\beta$ (and $\phi_0$ is then even and decreasing in $|x|$). Furthermore, $0<m_{\phi_0}<\beta$ and $F<0$ in $(0,m_{\phi_0}]$. All conditions of Theorem~\ref{th1} are then fulfilled.

Let now $u$ be a positive bounded solution of~\eqref{eq:ACN} satisfying~\eqref{localized}. As in~\eqref{defx0}, there is~$x_0\in\R$ such that, for every $t\in\R$, the function $x\mapsto u(t,x+x_0)$ is even in $x$ and decreasing in~$|x|$. Since $\phi:=\phi_0(\cdot-x_0)$ is the only solution of~\eqref{steady} which is symmetric with respect to~$x_0$, it follows from property~\eqref{conv1} of Theorem~\ref{th1} that
$$\|u(t,\cdot)-\phi\|_{L^\infty(\R)}\to0\hbox{ as $t\to-\infty$}.$$
Lastly, if alternative~(ii) holds in the conclusion of Theorem~\ref{th1}, then again by the uniqueness of the symmetric (with respect to $x_0$) solution $\phi$ of~\eqref{steady}, one has $\|u(t,\cdot)-\phi\|_{L^\infty(\R)}\to0$ as~$t\to+\infty$. It then follows from~\eqref{Eutl},~\eqref{Ephil} and~\eqref{Ephi} that the action $E[u(t,\cdot)]$ defined by~\eqref{defE} has the same limit $E[\phi]$ as $t\to\pm\infty$. From~\eqref{lyapounov}, one concludes that $u_t\equiv0$ in~$\R\times\R$, that is,~$u(t,x)\equiv\phi(x)$ in~$\R\times\R$. The proof of Corollary~\ref{cor1} is thereby complete.
\end{proof}

\begin{proof}[Proof of Corollary $\ref{cor3}$]
As in~\eqref{defx0}, there is a point $x_0\in\R^N$ such that, for every $t\in\R$, the function $x\mapsto u(t,x+x_0)$ is radially symmetric, and decreasing in $|x|$. From the assumptions made in Corollary~\ref{cor3}, the set of solutions of~\eqref{steady} which are radially symmetric with respect to the point $x_0$ is discrete. Since the $\alpha$-limit set of $u$ is non-empty, connected and made of solutions of~\eqref{steady} which are radially symmetric with respect to $x_0$ (from the proof of~\eqref{conv1} in Section~\ref{sec41}), it follows that there is $\phi\in\mathcal{E}$ such that
$$\|u(t,\cdot)-\phi\|_{L^\infty(\R^N)}\to0\hbox{ as $t\to-\infty$}.$$
In case alternative~(ii) of the conclusion of Theorem~\ref{th1} occurs, then there is $\phi'\in\mathcal{E}$ such that~$\|u(t,\cdot)-\phi'\|_{L^\infty(\R^N)}\to0$ as $t\to+\infty$. Furthermore, if $\phi=\phi'$, then, as in the above proof of Corollary~\ref{cor1}, one has $E[u(t,\cdot)]\to E[\phi]=E[\phi']$ as $t\to\pm\infty$, hence $u_t\equiv0$ and~$u(t,x)\equiv\phi(x)\equiv\phi'(x)$ in $\R\times\R^N$. 
\end{proof}


\section*{Acknowledgements}

This work has been carried out in the framework of Archim\`ede Labex of Aix-Marseille University. The project leading to this publication has received funding from Excellence Initiative of Aix-Marseille University~-~A*MIDEX, a French ``Investissements d'Avenir'' programme, and from the ANR NONLOCAL project (ANR-14-CE25-0013). The second author was partially supported by JSPS KAKENHI Grant Numbers JP16KT0022 and JP20H01816 and also thanks Aix-Marseille University for the warm hospitality during the visit. The authors are grateful to Thomas Giletti, Louis Jeanjean, Peter Pol\'a{\v{c}}ik, and Luca Rossi for helpful discussions and useful references.


\end{document}